\documentclass[11pt]{article}
\usepackage{ccicons}
\usepackage{amssymb,amsmath,amsfonts,amsthm,mathrsfs}
\usepackage{graphicx,graphics,psfrag,epsfig,calrsfs,cancel}
\usepackage[colorlinks=true, pdfstartview=FitV, linkcolor=blue, citecolor=red, urlcolor=blue]{hyperref}

\usepackage{caption,subfig,color}
\usepackage[dvipsnames]{xcolor}
\usepackage{tikz}
\usetikzlibrary{calc}

\usepackage{geometry}
\newtheorem{theorem}{Theorem}[section]

\newtheorem{lemma}[theorem]{Lemma}

\numberwithin{equation}{section}
\numberwithin{figure}{section}

\newcommand{\N}{\mathbb{N}}
\newcommand{\Z}{\mathbb{Z}}
\newcommand{\R}{\mathbb{R}}

\begin{document}

\title{The Neumann boundary condition \\
for the two-dimensional Lax-Wendroff scheme. II}

\author{Antoine {\sc Benoit}\thanks{Universit\'e du Littoral C\^ote d'Opale, Laboratoire de Math\'ematiques Pures et Appliqu\'ees, 
50 rue Ferdinand Buisson, CS 80699, 62228 Calais, France. Email: {\tt antoine.benoit@univ-littoral.fr}.} $\,$ \& Jean-Fran\c{c}ois 
{\sc Coulombel}\thanks{Institut de Math\'ematiques de Toulouse ; UMR5219, Universit\'e de Toulouse ; CNRS, F-31062 Toulouse 
Cedex 9, France. Email: {\tt jean-francois.coulombel@math.univ-toulouse.fr}. For the purpose of open access, the author will apply 
a CC-BY public copyright licence to any Author Accepted Manuscript version arising from this preliminary version.}}
\date{\today}
\maketitle

\begin{abstract}
We study the stability of a two-dimensional Lax-Wendroff scheme in a quarter-plane. Following our previous work \cite{BC1}, we aim 
here at adapting the energy method in order to study \emph{second order} extrapolation boundary conditions. We first show on the 
one-dimensional problem why modifying the energy is a necessity in order to obtain stability estimates. We then study the two-dimensional 
case and propose a modified energy as well as second order extrapolation boundary and corner conditions in order to maintain second 
order accuracy and stability of the whole scheme, including near the corner.
\end{abstract}


\bigskip

\noindent For the purpose of open access, this work is distributed under a Creative Commons Attribution | 4.0 International licence: 
\href{https://creativecommons.org/licenses/by/4.0/}{https://creativecommons.org/licenses/by/4.0/}.

\begin{flushleft}
\includegraphics{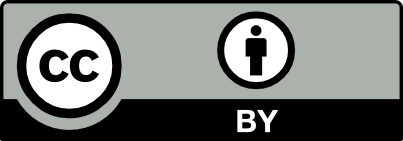}
\end{flushleft}

\paragraph{Notation.} For $d$ a positive integer and $\mathcal{J} \subset \Z^d$, we let $\ell^2(\mathcal{J};\R)$ denote the Hilbert space of 
real valued, square integrable sequences indexed by $\mathcal{J}$ and equipped with the norm:
$$
\forall \, u \in \ell^2(\mathcal{J};\R) \, ,\quad \| \, u \, \|_{\ell^2(\mathcal{J})}^2 \, := \, \sum_{j \in \mathcal{J}} \, u_j^2 \, .
$$
The corresponding scalar product is denoted $\langle \, \, ; \, \, \rangle_{\ell^2(\mathcal{J})}$. We mainly focus below on the case $d=2$ 
but we shall also encounter the case $d=1$.

\section{Introduction}
\label{section1}

This article is a follow-up of our previous work \cite{BC1} where we have studied the so-called Lax-Wendroff scheme with a stabilizer in two 
space dimensions. This scheme was proposed in \cite{laxwendroff} to approximate solutions to symmetric hyperbolic systems. Previous 
stability studies for this scheme were based on Fourier analysis and therefore dealt with problems that were defined on the whole space or 
that considered periodic boundary conditions. In \cite{BC1}, we have shown that the \emph{energy method} was a successful technique for 
dealing with $\ell^2$-stability of the Lax-Wendroff scheme in two space dimensions. The energy method bypasses Fourier analysis and is 
therefore interesting if one wishes to deal with problems with boundary conditions. In \cite{BC1}, we were able to recover the optimal stability 
criterion in the whole space (the so-called Courant-Friedrichs-Lewy condition) and also to study \emph{first order} extrapolation boundary 
conditions for an outflow in the half-plane and in the quarter-plane. For the latter case, the analysis requires specifying an extrapolation 
condition at the corner which, up to our knowledge, was new.

Since the Lax-Wendroff scheme gives, at least formally, second order approximations of solutions to symmetric hyperbolic systems, first order 
extrapolation at the boundary might deteriorate the overall accuracy of the scheme. We thus aim here at studying \emph{second order} extrapolation 
boundary conditions in the outflow case, that is when the transport operator does not come with any boundary condition in the continuous setting. 
As evidenced in the one-dimensional case (see Section \ref{section2} below), the classical energy method does not predict stability for second 
order extrapolation, at least not in a straightforward way. Modifying the energy near the boundary is necessary to obtain stability estimates by 
energy arguments, and this is probably one of the very first examples of discrete summation by parts operators (see, e.g., 
\cite{Strand1,Strand2,Mattsson} and subsequent works). Let us note that stability estimates for any order of extrapolation at the boundary could 
also be derived through the more complete, though elaborate, GKS analysis, see e.g. \cite{goldberg,goldberg-tadmor}, but we wish to bypass 
this theory in order, for instance, to cover problems in a quarter-plane for which an analogous theory is still lacking. Our goal here is therefore 
to extend the procedure of devising a suitable energy functional for the Lax-Wendroff scheme to \emph{second order extrapolation} boundary 
conditions in two space dimensions. This is, to some extent, a prototype example for a ``high order'' boundary treatment in several space 
dimensions with a corner in the space domain, and we shall already see that the algebra becomes rather involved.

The plan of the article is as follows. In Section \ref{section2}, we introduce and quickly analyze a one-dimensional problem in order to motivate 
the necessity of modifying the energy functional to deal with second order extrapolation. Section \ref{section3} is the core of this article. We 
introduce the two-dimensional Lax-Wendroff scheme and the associated extrapolation conditions in a quarter-plane. We then state and prove 
our main result, namely Theorem \ref{thm1} below. The general methodology is the same as in \cite{BC1} so we shall feel free at some places 
to shorten the details and refer to this companion article. At last Section \ref{part-num} includes some numerical simulations and a discussion enlightening the theoretical result and more specifically the assumptions made to obtain such a result.

\section{The one-dimensional problem}
\label{section2}

This section being mostly a presentation of a motivating example, we feel free not to make the functional framework precise and keep the calculations 
at a rather formal level. We consider the outgoing transport equation in one space dimension:
\begin{equation}
\label{hcl-1d}
\begin{cases}
\partial_t u \, + \, a \, \partial_x u \, = \, 0 \, ,& t \ge 0 \, ,\, x \ge 0 \, ,\\
u_{|_{t=0}} \, = \, u_0 \, ,& 
\end{cases}
\end{equation}
where $a$ is a fixed \emph{negative} number, which explains why we do not consider any boundary condition on $\{ x=0 \}$. The unknown function 
$u$ in \eqref{hcl-1d} is assumed to be real valued. We consider a space step $\Delta x>0$ and a time step $\Delta t>0$; we then denote $\lambda 
:= \Delta t/\Delta x$ the so-called Courant-Friedrichs-Lewy (CFL in what follows) number. We then approximate the solution to \eqref{hcl-1d} by the 
Lax-Wendroff scheme:
\begin{equation}
\label{LW-1d}
u_j^{n+1} \, = \, u_j^n \, - \, \dfrac{\lambda \, a}{2} \, (u_{j+1}^n-u_{j-1}^n) \, + \, \dfrac{(\lambda \, a)^2}{2} \, (u_{j+1}^n-2\, u_j^n+u_{j-1}^n) \, ,\quad 
n \in \N \, ,\quad j \in \N \, ,
\end{equation}
with the initial condition:
$$
\forall \, j \in \N \, ,\quad u_j^0 \, := \, \dfrac{1}{\Delta x} \, \int_{j \, \Delta x}^{(j+1) \, \Delta x} u_0(y) \, {\rm d}y \, .
$$
In \eqref{LW-1d}, $u_j^n$ is meant to be an approximation of the solution $u$ to \eqref{hcl-1d} in the cell $[n \, \Delta t,(n+1) \, \Delta t) \times 
[j \, \Delta x,(j+1) \, \Delta x)$ for any $(n,j) \in \N \times \N$. The iteration \eqref{LW-1d} requires the knowledge of $u_{-1}^n$ in order to determine 
$u_0^{n+1}$. For $n\in\mathbb{N}$, the cell $[n\,\Delta t \, ,\,(n+1)\,\Delta t )\times (-\Delta x \, , \, 0)$ that corresponds to the index $j=-1$ is referred 
to below as a \emph{ghost cell} since it lies outside of the physical domain $\R_x^+$. We consider here a \emph{second order extrapolation} procedure 
in order to maintain, at least formally, second order accuracy of the whole numerical procedure up to the boundary:
\begin{equation}
\label{extrapolation-1d}
u_{-1}^n \, = \, 2 \, u_0^n -u_1^n \, ,\quad n \in \N \, .
\end{equation}
The scheme \eqref{LW-1d}, \eqref{extrapolation-1d} then determines the sequence $(u_j^n)_{j \in \N}$ inductively with respect to $n$.

Multiplying the interior equation of \eqref{hcl-1d} by $u$ and integrating with respect to $x$ over $\R^+$, we obtain the energy inequality:
\begin{equation}
\label{energy-1d}
\dfrac{{\rm d}}{{\rm d}t} \int_{\R^+} u(t,x)^2 \, {\rm d}x \, = \, a \, u(t,0)^2 \, \le \, 0 \, .
\end{equation}
We aim here at understanding whether the numerical scheme \eqref{LW-1d}, \eqref{extrapolation-1d} satisfies an analogous energy balance law 
at the discrete level. The calculations below can already be found in \cite{CLundquist} but we reproduce them briefly for the sake of completeness.

We start from the following decomposition that is a direct consequence\footnote{Such decompositions that incorporate dissipation and telescopic terms 
are derived and used in a systematic way in \cite{jfcfl} to which we refer for more details.} of \eqref{LW-1d}:
\begin{align*}
\forall \, j \in \N \, ,\quad (u_j^{n+1})^2 \, - \, (u_j^n)^2 \, =& \, -\dfrac{(\lambda \, a)^2 \, (1-(\lambda \, a)^2)}{4} \, (u_{j+1}^n-2\, u_j^n+u_{j-1}^n)^2 \\
& +\lambda \, a \, (u_{j-1}^n \, u_j^n -u_j^n \, u_{j+1}^n) \, + \, \dfrac{(\lambda \, a)^2}{2} \, \Big( (u_{j-1}^n)^2 -2 \, (u_j^n)^2 +(u_{j+1}^n)^2 \Big) \\
& + \, \dfrac{(\lambda \, a)^3}{2} \, \Big( (u_j^n-u_{j-1}^n)^2-(u_{j+1}^n-u_j^n)^2 \Big) \, .
\end{align*}
The first term on the right-hand side corresponds to the dissipation of the Lax-Wendroff scheme while the second and third lines are telescopic with 
respect to $j$ (they would not contribute if we would sum over $\Z$). We now sum with respect to $j \in \N$ and then use the boundary condition 
\eqref{extrapolation-1d} to obtain:
\begin{align}
\forall \, n \in \N \, ,\quad \sum_{j \in \N} \, (u_j^{n+1})^2 \, - \, \sum_{j \in \N} \, (u_j^n)^2 \, =& \, 
-\dfrac{(\lambda \, a)^2 \, (1-(\lambda \, a)^2)}{4} \, \sum_{j \in \N} \, (u_{j+1}^n-2\, u_j^n+u_{j-1}^n)^2 \label{energy-discret-1d} \\
& +\dfrac{\lambda \, a-1}{2} \, (u_0^n)^2 \, + \, \dfrac{1+\lambda \, a}{2} \, \Big( u_0^n -\lambda \, a \, (u_1^n-u_0^n) \Big)^2 \, .\notag
\end{align}
We assume that the space and time steps are chosen in such a way that the stability condition $\lambda \, |a| \in (0,1)$ holds and we recall that $a$ 
is negative. In that case, the right-hand side in \eqref{energy-discret-1d} first incorporates a non-positive term that corresponds to the interior dissipation 
of the Lax-Wendroff scheme. It also incorporates, in the second line of \eqref{energy-discret-1d}, a \emph{boundary term} that is a \emph{quadratic form} 
with respect to $(u_0^n,u_1^n)$. This boundary term mimics the right-hand side of \eqref{energy-1d} since the discrete normal derivative $u_1^n-u_0^n$ 
is meant to be small for smooth solutions and $u_0^n$ is meant to be close to $u(n\, \Delta t,0)$.

Unfortunately, the above quadratic form on the right-hand side of \eqref{energy-discret-1d} is not negative definite since $\lambda \, a-1$ is negative 
but $1+\lambda \, a$ is positive. The energy argument thus does not predict stability, at least not in this straightforward way. However, it can easily be 
modified to obtain a positive conclusion to the stability problem of the scheme \eqref{LW-1d}, \eqref{extrapolation-1d}. The idea, following \cite{Strand1} 
and many subsequent works, is to modify the energy functional close to the numerical boundary. Namely, a direct adaptation of the above energy 
argument gives the identity:
\begin{align*}
\dfrac{1}{2} \, (u_0^{n+1})^2 \, + \, \sum_{j \ge 1} \, (u_j^{n+1})^2 \, - \, \dfrac{1}{2} \, (u_0^n)^2 \, - \, \sum_{j \ge 1} \, (u_j^n)^2 
\, =& \, -\dfrac{(\lambda \, a)^2 \, (1-(\lambda \, a)^2)}{4} \, \sum_{j \ge 1} \, (u_{j+1}^n-2\, u_j^n+u_{j-1}^n)^2 \\
& +\dfrac{\lambda \, a}{2} \, (u_0^n)^2 \, + \, \dfrac{\lambda \, a}{2} \, \Big( u_0^n -\lambda \, a \, (u_1^n-u_0^n) \Big)^2 \, ,
\end{align*}
where the interior dissipation term is unchanged but the boundary term is now a negative definite quadratic form of $(u_0^n,u_1^n)$ since $a$ is 
negative and $\lambda$ is a positive number. In particular, under the CFL condition $\lambda \, |a| \in (0,1)$, the energy:
$$
\dfrac{1}{2} \, (u_0^n)^2 \, + \, \sum_{j \ge 1} \, (u_j^n)^2 
$$
is non-increasing with respect to the time index $n$ for any solution to \eqref{LW-1d}, \eqref{extrapolation-1d} with square integrable initial condition. 
The note \cite{CLundquist} explains where the $1/2$ coefficient comes from (other coefficients close to $1/2$ could be chosen).
\bigskip

We aim below at extending such a modified energy technique to the two-dimensional case with second order extrapolation conditions, which would be 
an extension of \eqref{extrapolation-1d}. In \cite{BC1}, we have considered first the half-plane geometry and then the quarter-plane in order to present 
the associated algebra with slowly increasing difficulty. Since our main motivation is to investigate boundary conditions in regions with corners, we only 
deal here with the quarter-plane and leave the case of the half-plane to the interested reader.

\section{The two-dimensional problem}
\label{section3}

\subsection{The main result}

We consider from now on the two-dimensional transport equation in the quarter-plane $\R^+ \times \R^+$:
\begin{equation}
\label{hcl}
\begin{cases}
\partial_t u \, + \, a \, \partial_x u \, + \, b \, \partial_y u \, = \, 0 \, ,& t \ge 0 \, ,\, (x,y) \in \R^+ \times \R^+ \, ,\\
u_{|_{t=0}} \, = \, u_0 \, ,& 
\end{cases}
\end{equation}
where $a,b$ are some given real \emph{negative} numbers. The initial condition $u_0$ in \eqref{hcl} belongs to the Lebesgue space $L^2 
(\R^+ \times \R^+;\R)$. We consider below a finite difference approximation of \eqref{hcl} that is defined as follows. Given some space steps 
$\Delta x, \Delta y >0$ in each spatial direction, and given a time step $\Delta t>0$, we introduce the ratios $\lambda := \Delta t/\Delta x$ and 
$\mu :=\Delta t/\Delta y$. In all what follows, the ratios $\lambda$ and $\mu$ are assumed to be fixed, meaning that they are given \textit{a priori} of 
the computations and are meant to be tuned in order to satisfy some stability requirements (the so-called Courant-Friedrichs-Lewy condition 
\cite{cfl}, later on referred to as the CFL condition). The solution $u$ to \eqref{hcl} is then approximated on the time-space domain 
$[n \, \Delta t,(n+1) \, \Delta t) \times [j \, \Delta x,(j+1) \, \Delta x) \times [k \, \Delta y,(k+1) \, \Delta y)$ by a real number $u_{j,k}^n$ for 
any $n \in \N$ and $(j,k) \in \N^2$. The discrete initial condition $u^0$ is defined for instance by taking the piecewise constant projection 
of $u_0$ in \eqref{hcl} on each cell, that is (see \cite{gko}):
$$
\forall \, (j,k) \in \N^2 \, ,\quad u^0_{j,k} \, := \, \dfrac{1}{\Delta x \, \Delta y} \, \int_{j \, \Delta x}^{(j+1) \, \Delta x} \, \int_{k \, \Delta y}^{(k+1) \, \Delta y} 
\, u_0(x,y) \, {\rm d}x \, {\rm d}y \, .
$$
This discrete initial condition satisfies:
$$
\sum_{(j,k) \in \N^2} \, \Delta x \, \Delta y \, (u^0_{j,k})^2 \, \le \, \| \, u_0 \, \|_{L^2(\R^+ \times \R^+)}^2 \, .
$$
It then remains to determine the $u_{j,k}^n$'s inductively with respect to $n$. The Lax-Wendroff scheme with a stabilizer reads (see \cite{laxwendroff}):
\begin{align}
u_{j,k}^{n+1} \, = \, u_{j,k}^n 
& - \, \dfrac{\lambda \, a}{2} \, \big( u_{j+1,k}^n-u_{j-1,k}^n \big) \, - \, \dfrac{\mu \, b}{2} \, \big( u_{j,k+1}^n-u_{j,k-1}^n \big) \notag \\
& + \, \dfrac{(\lambda \, a)^2}{2} \, \big( u_{j+1,k}^n-2\, u_{j,k}^n+u_{j-1,k}^n \big) \, + \, 
\dfrac{(\mu \, b)^2}{2} \, \big( u_{j,k+1}^n-2\, u_{j,k}^n+u_{j,k-1}^n \big) \notag \\
& + \, \dfrac{\lambda \, a \, \mu \, b}{4} \, \big( u_{j+1,k+1}^n-u_{j+1,k-1}^n-u_{j-1,k+1}^n+u_{j-1,k-1}^n \big) \label{LW} \\
& - \, \dfrac{(\lambda \, a)^2 + (\mu \, b)^2}{8} \, \big( u_{j+1,k+1}^n-2\, u_{j+1,k}^n+u_{j+1,k-1}^n \notag \\
& \qquad \qquad \qquad \qquad - \, 2\, u_{j,k+1}^n+4\, u_{j,k}^n-2\, u_{j,k-1}^n+u_{j-1,k+1}^n-2\, u_{j-1,k}^n+u_{j-1,k-1}^n \big) \, ,\notag
\end{align}
where $(j,k)$ belongs to $\N^2$. We refer to \cite{laxwendroff,gko} for alternative approximations of \eqref{hcl}.

Since the computation of $u_{j,k}^{n+1}$ requires the knowledge of all closest neighboring cell values $u_{j+j',k+k'}^n$, with $j',k' \in \{ -1,0,1 \}$, 
we need to prescribe the values of the discrete solution $u^n$ in the \emph{ghost cells}, which correspond to the values $u_{-1,\ell}^n$ and 
$u_{\ell,-1}^n$ with $\ell \in \N$, and to the value $u_{-1,-1}^n$. These ghost cells are depicted in red and green in Figure \ref{fig:maillage}. 
The interior cells are depicted in blue.

\begin{figure}[htbp]
\begin{center}
\begin{tikzpicture}[scale=1.5,>=latex]
\draw [ultra thin, dotted, fill=blue!20] (-3,-1.5) rectangle (3.5,2);
\draw [ultra thin, dotted, fill=red!20] (-3.5,-1.5) rectangle (-3,2);
\draw [ultra thin, dotted, fill=red!20] (-3,-2) rectangle (3.5,-1.5);
\draw [ultra thin, dotted, fill=green!20] (-3.5,-2) rectangle (-3,-1.5);
\draw [thin, dashed] (-3.5,-2) grid [step=0.5] (3.5,2);
\draw[thick,black,->] (-4,-1.5) -- (4,-1.5) node[below] {$x$};
\draw[thick,black,->] (-3,-2.2)--(-3,2.5) node[right] {$y$};
\draw (-3.7,-2.12) node[right]{$-\Delta y$};
\draw (-3.5,-0.85) node[right]{$\Delta y$};
\draw (-3.7,1.15) node[right, fill=red!20]{$k \Delta y$};
\draw [thin, dashed] (-3.5,1) grid [step=0.5] (-3,1.5);
\draw (-3.6,-1.68) node{$-\Delta x$};
\draw (-3.18,-1.68) node[right]{$0$};
\draw (-2.6,-1.68) node{$\Delta x$};
\draw (2.35,-1.68) node{$j \Delta x$};
\node (centre) at (-3.5,-1.5){$\times$};
\node (centre) at (-3,-2){$\times$};
\node (centre) at (-3,-1.5){$\times$};
\node (centre) at (-2.5,-1.5){$\times$};
\node (centre) at (2.5,-1.5){$\times$};
\node (centre) at (-3,1){$\times$};
\node (centre) at (-3,-1){$\times$};
\draw (2.77,1.25) node{$u_{j,k}^n$};
\end{tikzpicture}
\caption{The spatial grid for the quarter-plane. Interior cells appear in blue, the boundary ghost cells appear in red and the corner ghost cell 
appears in green. The value $u_{j,k}^n$ corresponds to the approximation in the cell $[n\,\Delta t \, , \, (n+1)\, \Delta t) \times 
[j \Delta x,(j+1) \Delta x) \times [k \Delta y,(k+1) \Delta y)$.}
\label{fig:maillage}
\end{center}
\end{figure}
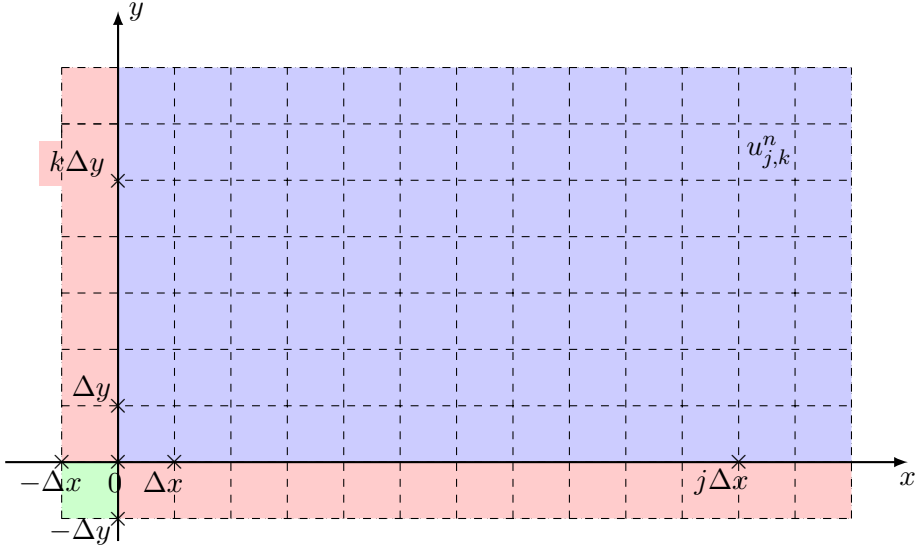

Extending the above one-dimensional analysis, we will impose second order extrapolation boundary conditions:
\begin{subequations}
\label{extrapolation}
\begin{align}
\label{extrapolationk}
\forall \, n \in \N \, ,\quad \forall \, k \in \N \, ,\quad u_{-1,k}^n \, &= \, 2 \, u_{0,k}^n \, - \, u_{1,k}^n \, ,\\
\label{extrapolationj}
\forall \, n \in \N \, ,\quad \forall \, j \in \N \, ,\quad u_{j,-1}^n \, &= \, 2 \, u_{j,0}^n \, - \, u_{j,1}^n \, ,
\end{align}
\end{subequations}
in conjunction with the numerical scheme \eqref{LW} for $(j,k) \in \N^2$ (that is, for interior values). It remains to define the corner cell value 
$u_{-1,-1}^n$. Following \cite{BC1} and trying, as in \eqref{extrapolation}, to have a symmetric treatment of both coordinates, we shall impose 
here the following procedure:
\begin{equation}
\label{corner}
\forall \, n \in \N \, ,\quad u_{-1,-1}^n \, = \, 4 \, u_{0,0}^n \, - \, 2 \, u_{1,0}^n  \, - \, 2 \, u_{0,1}^n  \, + \, u_{1,1}^n \, .
\end{equation}
Using \eqref{extrapolation}, this amounts equivalently to having:
$$
\dfrac{1}{2} \, \Big( u_{1,-1}^n -2\, u_{0,-1}^n +u_{-1,-1}^n \Big) \, + \, \dfrac{1}{2} \, \Big( u_{-1,1}^n -2\, u_{-1,0}^n +u_{-1,-1}^n \Big) \, = \, 0 \, .
$$
The equations \eqref{LW}, \eqref{extrapolation}, \eqref{corner} then define the sequence $(u_{j,k}^n)_{(j,k) \in \N^2}$ inductively with respect 
to $n \in \N$.

Our main result in this article is a stability estimate for solutions to \eqref{LW}, \eqref{extrapolation}, \eqref{corner}. For $u \in \ell^2(\N^2;\R)$, 
the standard norm is the one defined in the introduction of this article. However, it will be useful below to rely on the following \emph{equivalent} 
norm:
\begin{equation}
\label{def:norm}
\| \, u \, \|^2 \, := \, \sum_{j,k \ge 1} \, u_{j,k}^2 \, + \, \dfrac{1}{2} \, \sum_{k \ge 1} \, u_{0,k}^2 \, + \, \dfrac{1}{2} \, \sum_{j \ge 1} \, u_{j,0}^2 \, + \, 
\dfrac{1}{4} \, u_{0,0}^2 \, ,
\end{equation}
which is a two-dimensional analogue of the norm that we have shown to be useful in one space dimension (see Section \ref{section2}). 
The corresponding scalar product is denoted $\langle \, \, ; \, \, \rangle$ without referring to the space domain since it will be the underlying 
norm that we shall use from now on. Our main result is the following.

\begin{theorem}
\label{thm1}
Let $M>0$. Let the transport coefficients $a,b$ be negative, and let the associated CFL parameters $\lambda$, $\mu$ 
satisfy\footnote{A careful reading of the proof below shows that instead of \eqref{majorationCFL} one could assume the following bounds:
\begin{equation}\nonumber
\lambda \, |a| \, \le \, M \, \mu \, |b| \, \quad \text{and} \quad \mu \, |b| \, \le \, M' \, \lambda \, |a| \, ,
\end{equation}
with $M\neq M'$. We could thus obtain a non symmetric set (with respect to the first bisector) of admissible CFL parameters. The maximal radius 
$\varepsilon$ would depend on both $M$ and $M'$. We choose to expose the proof for $M=M'$ for the sake of simplicity.}:
\begin{equation}
\label{majorationCFL}
\lambda \, |a| \, \le \, M \, \mu \, |b| \, \quad \text{and} \quad \mu \, |b| \, \le \, M \, \lambda \, |a| \, .
\end{equation}
Then there exists some constant $\varepsilon>0$ that only depends on $M$, and there exists a numerical constant\footnote{We shall see for  
instance that $c=1/10$ is a suitable value but we have not tried to optimize the constant $c$ (nor $\varepsilon$).} $c>0$ such that, if $\lambda$, 
$\mu$ also satisfy:
$$
(\lambda \, a)^2 \, + \, (\mu \, b)^2 \, \le \, \varepsilon \, ,
$$
then for any $u^0 \in \ell^2(\N^2;\R)$, the solution to the numerical scheme \eqref{LW}, \eqref{extrapolation}, \eqref{corner} satisfies the energy 
estimate:
\begin{align*}
\|\, u^{n+1} \, \|^2 \, - \, \|\, u^n \, \|^2  \, &+ \, c \, (\lambda \, a)^2 \, \sum_{j,k \ge 1} \, (u_{j-1,k}^n-2\, u_{j,k}^n+u_{j+1,k}^n)^2 
\, + \, c \, (\mu \, b)^2 \, \sum_{j,k \ge 1} \, (u_{j,k-1}^n-2\, u_{j,k}^n+u_{j,k+1}^n)^2 \\
&+ \, c \, \lambda \, |a| \, \sum_{k \ge 0} \, (u_{0,k}^n)^2 \, + \, c \, \mu \, |b| \, \sum_{j \ge 0} \, (u_{j,0}^n)^2 \, \le \, 0 \, .
\end{align*}
\end{theorem}

Figure \ref{fig:CFL} illustrates the set of CFL parameters for which we obtain the stability of the Lax-Wendroff scheme with second 
order boundary and corner extrapolation.

\begin{figure}
\begin{center}
\begin{tikzpicture}[scale=1.5,>=latex]
\draw[thick,black,->] (-0.5,0) -- (4,0) node[below] {$\vert  \, a\, \vert \, \lambda$};
\draw[thick,black,->] (0,-0.5)--(0,4) node[right] {$\vert \, b\,\vert \, \mu $};
\draw[thick, black, fill=blue!20] (3,0) arc (0:90:3cm);
\fill[blue!20] (0,0)--(3,0)--(0,3);
\fill[red!60,thick,domain=15:75] plot ({2.5*cos(\x)}, {2.5*sin(\x)});
\draw[black,thick,dotted,domain=0:15] plot ({2.5*cos(\x)}, {2.5*sin(\x)});
\draw (2.5,0) node[below] {$\varepsilon$};
\draw (3,0) node[below] {$\frac{1}{\sqrt{2}}$};
\fill[red!60] (0,0)--(2.5*0.966,2.5*0.259)--(2.5*0.259,2.5*0.966);
\draw[black, thick, dotted] (2.5*0.966,2.5*0.259)--(3.5*0.966,3.5*0.259);
\draw[black, thick, dotted] (2.5*0.259,2.5*0.966)--(3.5*0.259,3.5*0.966);
\draw (3.5*0.966,3.5*0.259) node[right] {$\vert\,b\,\vert \mu\,=\, M\,\vert \, a\,\vert \,\lambda  $};
\draw (3.5*0.259,3.5*0.966) node[right] {$\vert\,a\,\vert \lambda\,=\, M\,\vert \, b\,\vert \,\mu  $};
\end{tikzpicture}
\caption{An illustration of admissible CFL parameters. The red area corresponds to the CFL parameters for which Theorem \ref{thm1} holds, 
the blue one to the optimal set of parameters (for which stability for the Cauchy problem holds).}
\label{fig:CFL}
\end{center}
\end{figure}
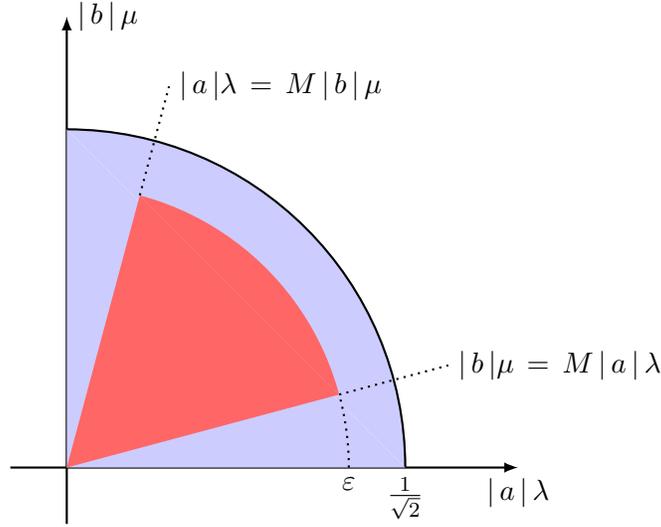

We remark that compared to \cite{BC1}, where we recovered the optimal set of parameters $(\vert \, a\,\vert\,\lambda)^2 \, + \, (\vert \, b\,\vert\,\mu)^2 
\, \leq \, \frac{1}{2}$ for the Cauchy problem and for first order extrapolation at the boundary, we now have some restrictions on the CFL parameters. 
These restrictions are of two types:
\begin{itemize}
\item A restriction of the maximal radius of the ball (the CFL parameters should be ``small enough'').
\item A restriction to a neighborhood of the first bisector (the CFL parameters should be ``comparable'').
\end{itemize}
These two restrictions are made in order to handle the much more involved algebra compared to \cite{BC1}. We do not claim that such restrictions 
are mandatory and maybe the energy method could be further refined in order to recover the maximal set of CFL parameters. Let us however indicate 
that, in our opinion, these restrictions are a little price to pay. Indeed reducing the maximal radius is not so restrictive and the second restriction (making 
the two ratios comparable) is rather natural from a practical point of view.

\subsection{Notation}

We follow the notation from \cite{BC1} and decompose the quantity $u_{j,k}^{n+1}$ in \eqref{LW} into three pieces:
$$
\forall \, (j,k) \in \N^2 \, ,\quad u_{j,k}^{n+1} \, = \, u_{j,k}^n \, - \, w_{j,k}^n \, + \, v_{j,k}^n \, ,
$$
where $v_{j,k}^n$ and $w_{j,k}^n$ are defined by:
\begin{subequations}
\label{defvwn}
\begin{align}
v_{j,k}^n \, :=& \, - \, \dfrac{\lambda \, a}{2} \, \big( u_{j+1,k}^n-u_{j-1,k}^n \big) \, - \, \dfrac{\mu \, b}{2} \, \big( u_{j,k+1}^n-u_{j,k-1}^n \big) \, ,\label{defvjkn} \\
w_{j,k}^n \, :=& \, - \, \dfrac{(\lambda \, a)^2}{2} \, \big( u_{j+1,k}^n-2\, u_{j,k}^n+u_{j-1,k}^n \big) \, - \, 
\dfrac{(\mu \, b)^2}{2} \, \big( u_{j,k+1}^n-2\, u_{j,k}^n+u_{j,k-1}^n \big) \notag \\
& - \, \dfrac{\lambda \, \mu \, a \, b}{4} \, \big( u_{j+1,k+1}^n-u_{j+1,k-1}^n-u_{j-1,k+1}^n+u_{j-1,k-1}^n \big) \label{defwjkn} \\
& + \, \dfrac{(\lambda \, a)^2 + (\mu \, b)^2}{8} \, \big( u_{j+1,k+1}^n-2\, u_{j+1,k}^n+u_{j+1,k-1}^n \notag \\
& \qquad \qquad \qquad \qquad - \, 2 \, u_{j,k+1}^n+4\, u_{j,k}^n-2\, u_{j,k-1}^n+u_{j-1,k+1}^n-2\, u_{j-1,k}^n+u_{j-1,k-1}^n \big) \, .\notag
\end{align}
\end{subequations}
We also use the shorthand notation $\alpha := \lambda \, a$ and $\beta := \mu \, b$. Both $\alpha$ and $\beta$ are negative real numbers.

The energy method in \cite{BC1} relies on symmetry or skew-symmetry properties of several finite difference operators. We thus introduce the 
following discrete first order partial derivatives and Laplacians:
\begin{align*}
&(D_{1,+}U)_{j,k} \, := \, U_{j+1,k} \, - \, U_{j,k} \, ,\quad (D_{1,-}U)_{j,k} \, := \, U_{j,k} \, - \, U_{j-1,k} \, ,\\
&(D_{2,+}U)_{j,k} \, := \, U_{j,k+1} \, - \, U_{j,k} \, ,\quad (D_{2,-}U)_{j,k} \, := \, U_{j,k} \, - \, U_{j,k-1} \, ,\\
&D_{1,0} \, := \, \dfrac{D_{1,+} +D_{1,-}}{2} \, ,\quad D_{2,0} \, := \, \dfrac{D_{2,+} +D_{2,-}}{2} \, ,\quad 
\Delta_1 \, := \, D_{1,+} \, D_{1,-} \, ,\quad \Delta_2 \, := \, D_{2,+} \, D_{2,-} \, .
\end{align*}
In order to keep the notation as simple as possible, we write below $D_{1,+}u_{j,k}$ rather than $(D_{1,+}u)_{j,k}$ and analogously for other operators. 
All above operators commute. Moreover the definitions allow us to rewrite \eqref{defvwn} as:
\begin{subequations}
\label{defvwn'}
\begin{align}
v^n \, :=& \, - \, \alpha \, D_{1,0} \, u^n \, - \, \beta \, D_{2,0} \, u^n \, ,\label{defvjkn'} \\
w^n \, :=& \, - \, 
\dfrac{\alpha^2}{2} \, \Delta_1 \, u^n \, - \, \dfrac{\beta^2}{2} \, \Delta_2 \, u^n \, - \, \alpha \, \beta \, D_{1,0} \, D_{2,0} \, u^n 
\, + \, \dfrac{\alpha^2 + \beta^2}{8} \, \Delta_1 \Delta_2 \, u^n \, .\label{defwjkn'}
\end{align}
\end{subequations}

Eventually, we follow the notation of \cite{BC1} and use the discrete set of indices $\mathbb{I} := \N^2$ for the \emph{interior} values of the numerical 
solution. It will be convenient below to use the notation $\mathring{\mathbb{I}} := \N^* \times \N^*$. Eventually, we denote $\mathbb{J} := (\{ -1 \} \cup \N)^2$ 
for the set of indices on which each sequence $u^n$ is defined (including the indices that correspond to the ghost cells). The underlying Hilbert space that 
corresponds to the norm in \eqref{def:norm} is the following set:
\begin{align*}
\mathbb{H} \, := \, \Big\{ u \in \ell^2(\mathbb{J};\R) \quad | \quad \forall \, k \in \N \, ,\quad u_{-1,k} \, &= \, 2 \, u_{0,k} \, - \, u_{1,k} \, , \\
\forall \, j \in \N \, ,\quad u_{j,-1} \, &= \, 2 \, u_{j,0} \, - \, u_{j,1} \, , \\
\text{\rm and } \quad u_{-1,-1} \, &= \, 4 \, u_{0,0} \, - \, 2 \, u_{1,0}  \, - \, 2 \, u_{0,1}  \, + \, u_{1,1} \, \Big\} \, .
\end{align*}
When equipped with the norm defined in \eqref{def:norm}, $\mathbb{H}$ becomes a Hilbert space, and the question we address in this article is mainly 
about understanding whether the numerical scheme defined by \eqref{LW} in $\mathbb{I}$, with $u^n \in \mathbb{H}$ for any $n \in \N$, yields a bounded 
sequence in $\mathbb{H}$. Let us observe that the sequences $v^n$ and $w^n$ in \eqref{defvwn'} are only defined on $\mathbb{I}$ and do not belong to 
$\mathbb{H}$. Nevertheless, we sometimes consider the norm in \eqref{def:norm} and its associated scalar product as acting on elements of $\mathbb{H}$ 
or on elements of $\ell^2(\mathbb{I};\R)$, see for instance Lemma \ref{lem1} below, since the norm in \eqref{def:norm} only involves those indices in 
$\mathbb{I}$ and do not involve the values in the ghost cells. We hope that this slight abuse will not create any confusion.

\subsection{Preliminary calculations}

The decomposition of $u^{n+1}$ gives the expression:
\begin{align}
\| \, u^{n+1} \, \|^2 \, - \, \| \, u^n \, \|^2 \, =& \, 2 \, \langle u^n ; v^n \rangle \, - \, 2 \,  \langle v^n ; w^n \rangle \notag \\
& \, + \, \| \, v^n \, \|^2 \, - \, 2 \,  \langle u^n ; w^n \rangle \, + \, \| \, w^n \, \|^2 \, .\label{bilan-energie}
\end{align}
The first two terms on the right-hand side are referred to below as the \emph{skew-symmetric} terms since they would not contribute on $\Z^2$ 
(see \cite{BC1}). The three other terms on the second line of the right-hand side are referred to below as the \emph{symmetric} terms. They will 
provide with the interior dissipation of the Lax-Wendroff scheme and will also give contributions both on the boundaries and at the corner.

We start with the expression of the two skew-symmetric terms.

\begin{lemma}
\label{lem1}
Let $a,b<0$ so that $\alpha,\, \beta <0$. Let $u^n \in \mathbb{H}$, and let the sequences $v^n,w^n$ be defined on the set of interior indices 
$\mathbb{I}$ by \eqref{defvwn'}. Then there holds:
\begin{subequations}
\label{termes-antisym1}
\begin{align}
2 \, \langle u^n;v^n\rangle \, =& - \, |\alpha| \, \sum_{k \ge 1} \, (u_{0,k}^n)^2 \, - \, |\beta| \, \sum_{j \ge 1} \, (u_{j,0}^n)^2 
\, - \, \dfrac{|\alpha|+|\beta|}{2} \, (u_{0,0}^n)^2 \, ,\label{lem1-antisym1} \\
- \, 2 \, \langle v^n ; w^n \rangle \, =& - \, \dfrac{|\alpha|^3}{2} \, \sum_{k \ge 1} (D_{1,+} u_{0,k}^n)^2 
\, - \, \dfrac{|\beta|^3}{2} \, \sum_{j \ge 1} (D_{2,+} u_{j,0}^n)^2 \notag \\
& - \, \dfrac{|\alpha| \, \beta^2}{2} \, \sum_{k \ge 1} (D_{2,+} u_{0,k}^n)^2 
\, - \, \dfrac{\alpha^2 \, |\beta|}{2} \, \sum_{j \ge 1} (D_{1,+} u_{j,0}^n)^2 \notag \\
& + \, \dfrac{|\alpha| \, \beta^2}{4} \, \sum_{k \ge 1} (\Delta_2 u_{0,k}^n)^2 
\, + \, \dfrac{\alpha^2 \, |\beta|}{4} \, \sum_{j \ge 1} (\Delta_1 u_{j,0}^n)^2 \notag \\
& - \, \alpha^2 \, |\beta| \, \sum_{k \ge 1} D_{2,0} u_{0,k}^n \, D_{1,+} u_{0,k}^n 
\, - \, |\alpha| \, \beta^2 \, \sum_{j \ge 1} D_{1,0} u_{j,0}^n \, D_{2,+} u_{j,0}^n \notag \\
& - \, |\alpha| \, \dfrac{\alpha^2+\beta^2}{8} \, \sum_{k \ge 1} (D_{1,+} D_{2,+}u_{0,k}^n)^2 
 \, - \,  |\beta| \, \dfrac{\alpha^2+\beta^2}{8} \, \sum_{j \ge 1} (D_{1,+} D_{2,+} u_{j,0}^n)^2 \label{lem1-antisym2} \\
& + \, |\beta| \, \dfrac{\alpha^2+\beta^2}{4} \, \sum_{k \ge 1} D_{2,0} u_{0,k}^n \, D_{1,+} \Delta_2 u_{0,k}^n 
\, + \, |\alpha| \, \dfrac{\alpha^2+\beta^2}{4} \, \sum_{j \ge 1} D_{1,0} u_{j,0}^n \, D_{2,+} \Delta_1 u_{j,0}^n \notag \\
& - \, \dfrac{|\alpha|^3}{4} \, (D_{1,+} u_{0,0}^n)^2 \, - \, \dfrac{|\beta|^3}{4} \, (D_{2,+} u_{0,0}^n)^2 
\, - \, (|\alpha|+|\beta|) \, \dfrac{\alpha^2+\beta^2}{8} \, (D_{1,+} D_{2,+} u_{0,0}^n)^2 \notag \\
& - \, \dfrac{\alpha^2 \, |\beta|}{2} \, (D_{1,+} u_{0,0}^n)^2 \, - \, \dfrac{|\alpha| \, \beta^2}{2} \, (D_{2,+} u_{0,0}^n)^2 
\, - \, \dfrac{|\alpha \, \beta|}{2} \, (|\alpha|+|\beta|) \, D_{1,+} u_{0,0}^n \, D_{2,+} u_{0,0}^n \, .\notag
\end{align}
\end{subequations}
\end{lemma}

\begin{proof}
$\bullet$ We start with the proof of \eqref{lem1-antisym1}. We first use definition \eqref{defvjkn'} and compute:
$$
2 \, \langle u^n;v^n\rangle \, = \, - 2 \, \alpha \, \langle u^n;D_{1,0} u^n\rangle \, - 2 \, \beta \, \langle u^n;D_{2,0} u^n\rangle \, ,
$$
and we now compute the first term on the right-hand side (the second one is analogous). An important observation for what follows 
is that the boundary condition \eqref{extrapolationj} gives:
$$
\forall \, n \in \N \, ,\quad \forall \, k \in \N \, ,\quad D_{1,0} u^n_{0,k} \, = \, \dfrac{u^n_{1,k}-u^n_{-1,k}}{2} \, = \, u^n_{1,k}-u^n_{0,k} 
\, = \, D_{1,+} u^n_{0,k} \, ,
$$
and, symmetrically:
$$
\forall \, n \in \N \, ,\quad \forall \, j \in \N \, ,\quad D_{2,0} u^n_{j,0} \, = \, D_{2,+} u^n_{j,0} \, .
$$
We thus have (see \eqref{def:norm} for the norm in $\mathbb{H}$ and its associated scalar product):
\begin{align*}
2 \, \langle u^n;D_{1,0} u^n\rangle =& \sum_{j,k \ge 1} u^n_{j,k} \, (u^n_{j+1,k}-u^n_{j-1,k}) \, + \, \sum_{k \ge 1} u^n_{0,k} \, (u^n_{1,k}-u^n_{0,k}) \\
&\, + \, \dfrac{1}{2} \, \sum_{j \ge 1} u^n_{j,0} \, (u^n_{j+1,0}-u^n_{j-1,0}) \, + \, \dfrac{1}{2} \, u^n_{0,0} \, (u^n_{1,0}-u^n_{0,0}) \\
=&\, - \, \sum_{k \ge 1} u^n_{1,k} \, u^n_{0,k} \, + \, \sum_{k \ge 1} u^n_{0,k} \, (u^n_{1,k}-u^n_{0,k}) 
\, - \, \dfrac{1}{2} \, u^n_{1,0} \, u^n_{0,0} \, + \, \dfrac{1}{2} \, u^n_{0,0} \, (u^n_{1,0}-u^n_{0,0}) \\
=&\, - \, \sum_{k \ge 1} (u^n_{0,k})^2  \, - \, \dfrac{1}{2} \, (u^n_{0,0})^2 \, . 
\end{align*}
Here we have used that the sums with respect to the first index $j$ are telescopic. Expression \eqref{lem1-antisym1} follows because 
$\alpha$ and $\beta$ are negative.
\bigskip

$\bullet$ We now turn to the proof of \eqref{lem1-antisym2}. We start from the definitions \eqref{defvwn'} and compute:
\begin{align}
- \, 2 \, \langle v^n;w^n\rangle \, =& \, - \, \alpha^3 \, \langle D_{1,0} u^n;\Delta_1 u^n \rangle \, - \, 
\beta^3 \, \langle D_{2,0} u^n;\Delta_2 u^n \rangle \notag \\
& \, - \, 2 \, \alpha^2 \, \beta \, \langle D_{1,0} u^n;D_{1,0} D_{2,0} u^n \rangle \, - \, 
2 \, \alpha \, \beta^2 \, \langle D_{2,0} u^n;D_{1,0} D_{2,0} u^n \rangle \notag \\
& \, - \, \alpha^2 \, \beta \, \langle \Delta_1 u^n;D_{2,0} u^n \rangle \, - \, \alpha \, \beta^2 \, \langle \Delta_2 u^n;D_{1,0} u^n \rangle 
\label{lem1-expression1} \\
& \, + \, \dfrac{\alpha^2+\beta^2}{4} \, \langle \alpha \, D_{1,0} u^n + \beta \, D_{2,0} u^n;\Delta_1 \Delta_2 u^n \rangle \, .\notag
\end{align}
We then compute each line on the right-hand side of \eqref{lem1-expression1} separately and, eventually, we combine them.

Observing that \eqref{extrapolationj} gives $\Delta_1 u^n_{0,k}=0$ for any $k \in \N$, we write $\Delta_1=D_{1,+} - D_{1,-}$ and recall the 
relation $D_{1,0}=(D_{1,+}+D_{1,-})/2$. We thus compute:
\begin{align*}
\langle D_{1,0} u^n;\Delta_1 u^n\rangle \, =& \, \sum_{j,k \ge 1} \, D_{1,0} u^n_{j,k} \, \Delta_1 u^n_{j,k} \, + \, 
\dfrac{1}{2} \, \sum_{j \ge 1} \, D_{1,0} u^n_{j,0} \, \Delta_1 u^n_{j,0} \\ 
=&\, \dfrac{1}{2} \, \sum_{j,k \ge 1} \, (D_{1,+} u^n_{j,k})^2 -(D_{1,-} u^n_{j,k})^2 \, + \, 
\dfrac{1}{4} \, \sum_{j \ge 1} \, (D_{1,+} u^n_{j,0})^2 -(D_{1,-} u^n_{j,0})^2 \\
=&\, - \, \dfrac{1}{2} \, \sum_{k \ge 1} \, (D_{1,+} u^n_{0,k})^2 \, - \, \dfrac{1}{4} \, (D_{1,+} u^n_{0,0})^2 \, ,
\end{align*}
and similarly for the scalar product $\langle D_{2,0} u^n;\Delta_2 u^n\rangle$. We thus have:
\begin{multline}
\label{lem1-ligne1}
- \, \alpha^3 \, \langle D_{1,0} u^n;\Delta_1 u^n \rangle \, - \, \beta^3 \, \langle D_{2,0} u^n;\Delta_2 u^n \rangle \\
= \, - \, \dfrac{|\alpha|^3}{2} \, \sum_{k \ge 1} \, (D_{1,+} u^n_{0,k})^2  \, - \, \dfrac{|\beta|^3}{2} \, \sum_{j \ge 1} \, (D_{2,+} u^n_{j,0})^2 
\, - \, \dfrac{|\alpha|^3}{4} \, (D_{1,+} u^n_{0,0})^2 \, - \, \dfrac{|\beta|^3}{4} \, (D_{2,+} u^n_{0,0})^2 \, .
\end{multline}
These terms constitute the first line on the right-hand side of \eqref{lem1-antisym2} as well as the first two terms in the seventh line of 
\eqref{lem1-antisym2}.

We now turn to the second line on the right-hand side of \eqref{lem1-expression1}. We first observe that the extrapolation condition \eqref{corner}
gives the relation $D_{1,0} D_{2,0} u^n_{0,0} \, = \, D_{1,+} D_{2,+} u^n_{0,0}$. Using the extrapolation conditions \eqref{extrapolation} and the 
definition of the operator $D_{2,+}$, we therefore compute:
\begin{align*}
\langle D_{1,0} u^n;D_{1,0} D_{2,0} u^n \rangle \, =& \, 
\dfrac{1}{2} \, \sum_{j,k \ge 1} \, (D_{1,0} u^n_{j,k} \, D_{1,0} u^n_{j,k+1} \, - \, D_{1,0} u^n_{j,k} \, D_{1,0} u^n_{j,k-1}) 
\, + \, \dfrac{1}{2} \, \sum_{j \ge 1} \, D_{1,0} u^n_{j,0} \, D_{1,0} D_{2,+} u^n_{j,0} \\
&+ \dfrac{1}{4} \, \sum_{k \ge 1} \, (D_{1,+} u^n_{0,k} \, D_{1,+} u^n_{0,k+1} \, - \, D_{1,+} u^n_{0,k} \, D_{1,+} u^n_{0,k-1}) 
\, + \dfrac{1}{4} \, D_{1,+} u^n_{0,0} \, D_{1,+} D_{2,+} u^n_{0,0} \\
=& \, - \, \dfrac{1}{2} \, \sum_{j \ge 1} \, D_{1,0} u^n_{j,0} \, D_{1,0} u^n_{j,1} \, + \, 
\dfrac{1}{2} \, \sum_{j \ge 1} \, D_{1,0} u^n_{j,0} \, D_{2,+} D_{1,0} u^n_{j,0} \\
&\, - \, \dfrac{1}{4} \, D_{1,+} u^n_{0,0} \, D_{1,+} u^n_{0,1} \, + \, \dfrac{1}{4} \, D_{1,+} u^n_{0,0} \, D_{2,+} D_{1,+} u^n_{0,0} \\
=& \, - \, \dfrac{1}{2} \, \sum_{j \ge 1} \, (D_{1,0} u^n_{j,0})^2 \, - \, \dfrac{1}{4} \, (D_{1,+} u^n_{0,0})^2 \, .
\end{align*}
We then use the following formula that is valid for any $\ell^2$ sequence on $\N$ (see \cite[Lemma 3.2]{BC1} for a similar computation):
\begin{equation}
\label{formule1d}
\sum_{j \ge 1} \, (D_{1,0} U_j)^2 \, + \, \dfrac{1}{4} \, \sum_{j \ge 1} \, (\Delta_1 U_j)^2 = 
\dfrac{1}{2} \, \sum_{j \ge 1} \, (D_{1,+} U_j)^2 \, + \, \dfrac{1}{2} \, \sum_{j \ge 1} \, (D_{1,-} U_j)^2 = 
\sum_{j \ge 1} \, (D_{1,+} U_j)^2 \, + \, \dfrac{1}{2} \, (D_{1,+} U_0)^2 \, ,
\end{equation}
and we thus get:
\begin{equation*}
\langle D_{1,0} u^n;D_{1,0} D_{2,0} u^n \rangle \, = \, - \, \dfrac{1}{2} \, \sum_{j \ge 1} \, (D_{1,+} u^n_{j,0})^2 
\, + \, \dfrac{1}{8} \, \sum_{j \ge 1} \, (\Delta_1 u^n_{j,0})^2 \, - \, \dfrac{1}{2} \, (D_{1,+} u^n_{0,0})^2 \, ,
\end{equation*}
with, of course, a similar expression for the other scalar product in the second line of \eqref{lem1-expression1}. Summarizing, we have 
obtained the expression:
\begin{align}
- \, 2 \, \alpha \, \beta^2 \, \langle D_{2,0} u^n;D_{1,0} D_{2,0} u^n \rangle \, -& \, 2 \, \alpha^2 \, \beta \, \langle D_{1,0} u^n;D_{1,0} D_{2,0} u^n \rangle 
\notag \\
= & - \, |\alpha| \, \beta^2 \, \sum_{k \ge 1} \, (D_{2,+} u^n_{0,k})^2 \, - \, \alpha^2 \, |\beta| \, \sum_{j \ge 1} \, (D_{1,+} u^n_{j,0})^2 \notag \\
& + \, \dfrac{|\alpha| \, \beta^2}{4} \, \sum_{k \ge 1} \, (\Delta_2 u^n_{0,k})^2 \, + \, \dfrac{\alpha^2 \, |\beta|}{4} \, \sum_{j \ge 1} \, (\Delta_1 u^n_{j,0})^2 
\label{lem1-ligne2} \\
& - \, |\alpha| \, \beta^2 \, (D_{2,+} u^n_{0,0})^2 \, - \, \alpha^2 \, |\beta| \, (D_{1,+} u^n_{0,0})^2 \, .\notag
\end{align}
These terms contribute to the second line and give the third line of \eqref{lem1-antisym2}. They will also contribute in the first two terms of the last 
line of \eqref{lem1-antisym2}.

We now turn to the third line on the right-hand side of \eqref{lem1-expression1}. Recalling that $\Delta_1u_{0,k}^n$ vanishes for any $k \in \N$, 
see \eqref{extrapolationj}, we have:
$$
\langle \Delta_1 u^n;D_{2,0} u^n \rangle \, = \, \sum_{j,k \ge 1} \, \Delta_1 u^n_{j,k} \, D_{2,0} u^n_{j,k} \, + \, 
\dfrac{1}{2} \, \sum_{j \ge 1} \, \Delta_1 u^n_{j,0} \, D_{2,+} u^n_{j,0} \, ,
$$
and we now use the discrete integration by parts formula:
\begin{equation}
\label{ipp1d}
\sum_{j \ge 1} \, (\Delta_1 U_j) \, V_j \, = \, - \, \sum_{j \ge 1} \, (D_{1,+} U_j) \, D_{1,+} V_j \, - \, (D_{1,+} U_0) \, V_1 \, ,
\end{equation}
which yields:
\begin{align*}
\langle \Delta_1 u^n;D_{2,0} u^n \rangle \, =& \, - \, \sum_{j,k \ge 1} \, D_{1,+} u^n_{j,k} \, D_{1,+} D_{2,0} u^n_{j,k} \, - \, 
\sum_{k \ge 1} \, D_{1,+} u^n_{0,k} \, D_{2,0} u^n_{1,k} \\
& \, - \, \dfrac{1}{2} \, \sum_{j \ge 1} \, D_{1,+} u^n_{j,0} \, D_{1,+} D_{2,+} u^n_{j,0} \, - \, \dfrac{1}{2} \, D_{1,+} u^n_{0,0} \, D_{2,+} u^n_{1,0} \, .
\end{align*}
The first sum on the right-hand side is telescopic with respect to $k$ and it partially simplifies with the third term on the right-hand side (the sum 
with respect to the index $j$ only). We get:
\begin{align*}
\langle \Delta_1 u^n;D_{2,0} u^n \rangle \, =& \, \dfrac{1}{2} \, \sum_{j \ge 1} \, (D_{1,+} u^n_{j,0})^2 
\, - \, \sum_{k \ge 1} \, D_{1,+} u^n_{0,k} \, D_{2,0} u^n_{1,k} \, - \, \dfrac{1}{2} \, D_{1,+} u^n_{0,0} \, D_{2,+} u^n_{1,0} \\
=& \, \dfrac{1}{2} \, \sum_{j \ge 1} \, (D_{1,+} u^n_{j,0})^2 \, - \, \sum_{k \ge 1} \, D_{1,+} u^n_{0,k} \, D_{2,0} u^n_{0,k} \\
& \, - \, \sum_{k \ge 1} \, D_{1,+} u^n_{0,k} \, D_{2,0} D_{1,+} u^n_{0,k} \, - \, \dfrac{1}{2} \, D_{1,+} u^n_{0,0} \, D_{2,+} D_{1,+} u^n_{0,0} 
\, - \, \dfrac{1}{2} \, D_{1,+} u^n_{0,0} \, D_{2,+} u^n_{0,0} \\
=& \, \dfrac{1}{2} \, \sum_{j \ge 1} \, (D_{1,+} u^n_{j,0})^2 \, - \, \sum_{k \ge 1} \, D_{1,+} u^n_{0,k} \, D_{2,0} u^n_{0,k} 
\, + \, \dfrac{1}{2} \, (D_{1,+} u^n_{0,0})^2 \, - \, \dfrac{1}{2} \, D_{1,+} u^n_{0,0} \, D_{2,+} u^n_{0,0} \, .
\end{align*}
Summing with the analogous term in the third line of \eqref{lem1-expression1}, this gives the contribution:
\begin{align}
- \, \alpha^2 \, \beta \, \langle \Delta_1 u^n;D_{2,0} u^n \rangle \, &- \, \alpha \, \beta^2 \, \langle \Delta_2 u^n;D_{1,0} u^n \rangle \notag \\
=& \, \dfrac{\alpha^2 \, |\beta|}{2} \, \sum_{j \ge 1} \, (D_{1,+} u^n_{j,0})^2 \, + \, \dfrac{|\alpha| \, \beta^2}{2} \, \sum_{k \ge 1} \, (D_{2,+} u^n_{0,k})^2 \notag \\
&\, - \, \alpha^2 \, |\beta| \, \sum_{k \ge 1} \, D_{2,0} u_{0,k}^n \, D_{1,+} u^n_{0,k} 
\, - \, |\alpha| \, \beta^2 \, \sum_{j \ge 1} \, D_{1,0} u_{j,0}^n \, D_{2,+} u^n_{j,0} \label{lem1-ligne3} \\
&\, + \, \dfrac{\alpha^2 \, |\beta|}{2} \, (D_{1,+} u^n_{0,0})^2 \, + \, \dfrac{|\alpha| \, \beta^2}{2} \, (D_{2,+} u^n_{0,0})^2 
\, - \, \dfrac{|\alpha| \, |\beta|}{2} \, (|\alpha| + |\beta|) \, D_{1,+} u^n_{0,0} \, D_{2,+} u^n_{0,0} \, .\notag
\end{align}
These terms give the final contribution in the second line of \eqref{lem1-antisym2}. They also give the fourth line of \eqref{lem1-antisym2}.  
Finally these terms give the final contribution in the eighth line of \eqref{lem1-antisym2}.

We now turn to the fourth and last line on the right-hand side of \eqref{lem1-expression1}. We first observe that the boundary conditions 
\eqref{extrapolation} and \eqref{corner} imply that the quantity $\Delta_1 \Delta_2 u^n_{j,k}$ vanishes whenever $j$ or $k$ (or both) is zero. 
In particular, \eqref{extrapolation} and \eqref{corner} imply $\Delta_1 \Delta_2 u^n_{0,0}=0$. We thus have:
$$
\langle D_{1,0} u^n;\Delta_1 \Delta_2 u^n \rangle \, = \, \sum_{j,k \ge 1} \, D_{1,0} u^n_{j,k} \, \Delta_1 \Delta_2 u^n_{j,k} \, ,
$$
and we then perform a discrete integration by parts with respect to $k$ (see \eqref{ipp1d}) to get:
\begin{align*}
\langle D_{1,0} u^n;\Delta_1 \Delta_2 u^n \rangle \, =& \, - \, \sum_{j,k \ge 1} \, D_{1,0} D_{2,+} u^n_{j,k} \, \Delta_1 D_{2,+} u^n_{j,k} 
\, - \, \sum_{j \ge 1} \, D_{1,0} u^n_{j,1} \, \Delta_1 D_{2,+} u^n_{j,0} \\
=& \, \dfrac{1}{2} \, \sum_{k \ge 1} \, (D_{1,+} D_{2,+} u^n_{0,k})^2 \, - \, \sum_{j \ge 1} \, D_{1,0} u^n_{j,1} \, \Delta_1 D_{2,+} u^n_{j,0} \\
=& \, \dfrac{1}{2} \, \sum_{k \ge 1} \, (D_{1,+} D_{2,+} u^n_{0,k})^2 \\
&- \, \sum_{j \ge 1} \, D_{1,0} D_{2,+} u^n_{j,0} \, \Delta_1 D_{2,+} u^n_{j,0} \, - \, \sum_{j \ge 1} \, D_{1,0} u^n_{j,0} \, \Delta_1 D_{2,+} u^n_{j,0} \\
=& \, \dfrac{1}{2} \, \sum_{k \ge 1} \, (D_{1,+} D_{2,+} u^n_{0,k})^2 \, + \, \dfrac{1}{2} \, (D_{1,+} D_{2,+} u^n_{0,0})^2 
 \, - \, \sum_{j \ge 1} \, D_{1,0} u^n_{j,0} \, \Delta_1 D_{2,+} u^n_{j,0} \, .
\end{align*}
We thus get the final contribution:
\begin{align}
\dfrac{\alpha^2+\beta^2}{4} \, \langle \alpha \, D_{1,0} u^n +& \beta \, D_{2,0} u^n;\Delta_1 \Delta_2 u^n \rangle \notag \\
=& \, - \, |\alpha| \, \dfrac{\alpha^2+\beta^2}{8} \, \sum_{k \ge 1} \, (D_{1,+} D_{2,+} u^n_{0,k})^2 
 \, - \, |\beta| \, \dfrac{\alpha^2+\beta^2}{8} \, \sum_{j \ge 1} \, (D_{1,+} D_{2,+} u^n_{j,0})^2 \label{lem1-ligne4} \\
& + \, |\alpha| \, \dfrac{\alpha^2+\beta^2}{4} \, \sum_{j \ge 1} \, D_{1,0} u^n_{j,0} \, \Delta_1 D_{2,+} u^n_{j,0} 
\, + \, |\beta| \, \dfrac{\alpha^2+\beta^2}{4} \, \sum_{k \ge 1} \, D_{2,0} u^n_{0,k} \, \Delta_2 D_{1,+} u^n_{0,k} \notag \\
& - \, (|\alpha|+|\beta|) \, \dfrac{\alpha^2+\beta^2}{8} \, (D_{1,+} D_{2,+} u^n_{0,0})^2 \, .\notag
\end{align}
These terms give the fifth and sixth lines of \eqref{lem1-antisym2} and the last term in the seventh line. It now only remains to collect the 
contributions in \eqref{lem1-ligne1}, \eqref{lem1-ligne2}, \eqref{lem1-ligne3} and \eqref{lem1-ligne4} to obtain the relation \eqref{lem1-antisym2}.
\end{proof}

\noindent We now explain how to derive the expression of the first symmetric term in \eqref{bilan-energie}.

\begin{lemma}
\label{lem2}
Let $a,b<0$. Let $u^n \in \mathbb{H}$, and let the sequences $v^n,w^n$ be defined on the set of interior indices $\mathbb{I}$ by \eqref{defvwn'}. 
Then there holds:
\begin{align}
\| \, v^n \, \|^2 \, - \, 2 \, \langle u^n;w^n\rangle \, =& \, - \, \dfrac{\alpha^2}{4} \, \| \Delta_1 u^n \|_{\ell^2(\mathring{\mathbb{I}})}^2 
\, - \, \dfrac{\beta^2}{4} \, \| \Delta_2 u^n \|_{\ell^2(\mathring{\mathbb{I}})}^2  \, - \, \dfrac{\alpha^2+\beta^2}{16} \, \Big( 
\| D_{1,-} D_{2,-} u^n \|_{\ell^2(\mathring{\mathbb{I}})}^2 \notag \\
& \quad \, + \, \| D_{1,-} D_{2,+} u^n \|_{\ell^2(\mathring{\mathbb{I}})}^2 \, + \, \| D_{1,+} D_{2,-} u^n \|_{\ell^2(\mathring{\mathbb{I}})}^2 
\, + \, \| D_{1,+} D_{2,+} u^n \|_{\ell^2(\mathring{\mathbb{I}})}^2 \Big) \notag \\
& \, - \, \dfrac{\alpha^2}{8} \, \sum_{j \ge 1} (\Delta_1 u_{j,0}^n)^2 \, - \, \dfrac{\beta^2}{8} \, \sum_{k \ge 1} (\Delta_2 u_{0,k}^n)^2 \notag \\
& \, - \, \alpha^2 \, \sum_{k \ge 1} u_{0,k}^n \, D_{1,+} u_{0,k}^n \, - \, \beta^2 \, \sum_{j \ge 1} u_{j,0}^n \, D_{2,+} u_{j,0}^n \label{lem2-sym1} \\
& \, - \, \dfrac{(\alpha^2+\beta^2)}{8} \, \sum_{j \ge 1} (D_{1,+} D_{2,+} u_{j,0}^n)^2  \, - \, \dfrac{(\alpha^2+\beta^2)}{8} \, 
\sum_{k \ge 1} (D_{1,+} D_{2,+} u_{0,k}^n)^2 \notag \\
& \, - \, \dfrac{(\alpha^2+\beta^2)}{4} \, \sum_{j \ge 1} D_{1,+} u_{j,0}^n \, D_{1,+} D_{2,+} u_{j,0}^n 
\, - \, \dfrac{(\alpha^2+\beta^2)}{4} \, \sum_{k \ge 1} D_{2,+} u_{0,k}^n \, D_{1,+} D_{2,+} u_{0,k}^n \notag \\
& \, + \, |\alpha \, \beta| \, (u_{0,0}^n)^2 \, - \, \dfrac{\alpha^2}{2} \, u_{0,0}^n \, D_{1,+} u_{0,0}^n \, - \, \dfrac{\beta^2}{2} \, u_{0,0}^n \, D_{2,+} u_{0,0}^n 
\notag \\
& \, - \, \dfrac{(\alpha^2+\beta^2)}{4} \, \Big( u_{0,0}^n + D_{1,+} u_{0,0}^n + D_{2,+} u_{0,0}^n \Big) \, D_{1,+} D_{2,+} u_{0,0}^n \notag \\
& \, - \, \dfrac{3 \, (\alpha^2+\beta^2)}{16} \, (D_{1,+} D_{2,+} u_{0,0}^n)^2 \, ,\notag
\end{align}
where we recall the notation $\mathring{\mathbb{I}}=\N^* \times \N^*$.
\end{lemma}

\begin{proof}
We start from the definitions \eqref{defvwn'} and compute:
\begin{align}
\| \, v^n \, \|^2 \, - \, 2 \, \langle u^n;w^n\rangle \, =& \, \alpha^2 \, \Big( \| D_{1,0} u^n \|^2 \, + \, \langle u^n;\Delta_1 u^n \rangle \Big) 
\, + \, \beta^2 \, \Big( \| D_{2,0} u^n \|^2 \, + \, \langle u^n;\Delta_2 u^n \rangle \Big) \notag \\
& \, + \, 2 \, \alpha \, \beta \, \Big( \langle D_{1,0} u^n;D_{2,0} u^n \rangle \, + \, \langle u^n;D_{1,0} D_{2,0} u^n \rangle \Big) \label{lem2-expression} \\
& \, - \, \dfrac{\alpha^2 + \beta^2}{4} \, \langle u^n;\Delta_1 \Delta_2 u^n \rangle \, .\notag
\end{align}
In order to simplify the first line on the right-hand side of \eqref{lem2-expression}, we use the telescopic formula (see \cite{BC1}):
$$
\dfrac{(U_{\ell+1}-U_{\ell-1})^2}{4} \, + \, U_\ell \, (U_{\ell+1}-2 \, U_\ell+U_{\ell-1}) \, = \, -\dfrac{(U_{\ell+1}-2\, U_\ell+U_{\ell-1})^2}{4} 
\, + \, \dfrac{1}{2} \, (U_{\ell+1}^2-U_\ell^2) \, - \, \dfrac{1}{2} \, (U_\ell^2-U_{\ell-1}^2) \, .
$$
We thus get:
$$
\sum_{j,k \ge 1} \, (D_{1,0} u^n_{j,k})^2 \, + \, \sum_{j,k \ge 1} \, u^n_{j,k} \, \Delta_1 u^n_{j,k} \, = \, 
- \, \dfrac{1}{4} \, \| \Delta_1 u^n \|_{\ell^2(\mathring{\mathbb{I}})}^2 \, - \,  \dfrac{1}{2} \, \sum_{k \ge 1} \, (u_{1,k}^n)^2 -(u_{0,k}^n)^2 \, ,
$$
and we also get a similar expression on one side of the boundary (since the tangential index $k$ is a mere parameter in this calculation):
$$
\sum_{j \ge 1} \, (D_{1,0} u^n_{j,0})^2 \, + \, \sum_{j \ge 1} \, u^n_{j,0} \, \Delta_1 u^n_{j,0} \, = \, - \, \dfrac{1}{4} \, \sum_{j \ge 1} \, (\Delta_1 u^n_{j,0})^2 
\, - \, \dfrac{1}{2} \, (u_{1,0}^n)^2 \, + \, \dfrac{1}{2} \, (u_{0,0}^n)^2 \, .
$$
Combining the previous two equalities and recalling that we have $\Delta_1 u^n_{0,k}=0$ for any $k \in \N$, we get:
\begin{align*}
\| D_{1,0} u^n \|^2 \, + \, \langle u^n;\Delta_1 u^n \rangle \, =& \, - \, \dfrac{1}{4} \, \| \Delta_1 u^n \|_{\ell^2(\mathring{\mathbb{I}})}^2 
 \, - \dfrac{1}{8} \, \sum_{j \ge 1} \, (\Delta_1 u^n_{j,0})^2 
\, + \dfrac{1}{2} \, \sum_{k \ge 1} \, (D_{1,+} u_{0,k}^n)^2 - (u_{1,k}^n)^2 + (u_{0,k}^n)^2 \\
& \, + \, \dfrac{1}{4} \, (D_{1,+} u_{0,0}^n)^2 \, - \, \dfrac{1}{4} \, (u_{1,0}^n)^2 \, + \, \dfrac{1}{4} \, (u_{0,0}^n)^2 \\
=& \, - \dfrac{1}{4} \, \| \Delta_1 u^n \|_{\ell^2(\mathring{\mathbb{I}})}^2 \, - \dfrac{1}{8} \, \sum_{j \ge 1} \, (\Delta_1 u^n_{j,0})^2 
\, - \sum_{k \ge 1} \, u_{0,k}^n \, D_{1,+} u_{0,k}^n \, - \dfrac{1}{2} \, u_{0,0}^n \, D_{1,+} u_{0,0}^n \, .
\end{align*}
We thus obtain the expression of the first line on the right-hand side of \eqref{lem2-expression}:
\begin{align}
\alpha^2 \, \Big( \| D_{1,0} u^n \|^2 \, +& \, \langle u^n;\Delta_1 u^n \rangle \Big) 
\, + \, \beta^2 \, \Big( \| D_{2,0} u^n \|^2 \, + \, \langle u^n;\Delta_2 u^n \rangle \Big) \notag \\
=& \, - \, \dfrac{\alpha^2}{4} \, \| \Delta_1 u^n \|_{\ell^2(\mathring{\mathbb{I}})}^2 
\, - \, \dfrac{\beta^2}{4} \, \| \Delta_2 u^n \|_{\ell^2(\mathring{\mathbb{I}})}^2 
\, - \, \dfrac{\alpha^2}{8} \, \sum_{j \ge 1} (\Delta_1 u_{j,0}^n)^2 \, - \, \dfrac{\beta^2}{8} \, \sum_{k \ge 1} (\Delta_2 u_{0,k}^n)^2 \label{lem2-relation1} \\
& \, - \, \alpha^2 \, \sum_{k \ge 1} u_{0,k}^n \, D_{1,+} u_{0,k}^n \, - \, \beta^2 \, \sum_{j \ge 1} u_{j,0}^n \, D_{2,+} u_{j,0}^n 
\, - \, \dfrac{\alpha^2}{2} \, u_{0,0}^n \, D_{1,+} u_{0,0}^n \, - \, \dfrac{\beta^2}{2} \, u_{0,0}^n \, D_{2,+} u_{0,0}^n \, .\notag
\end{align}
These terms give the first two terms of the first line, the third and fourth lines and the last two terms of the seventh line in \eqref{lem2-sym1}.

We now turn to the second line on the right-hand side of the decomposition \eqref{lem2-expression}. We use the telescopic formula:
$$
D_{1,0} u^n_{j,k} \, D_{2,0} u^n_{j,k} \, + \, u^n_{j,k} \, D_{1,0} D_{2,0} u^n_{j,k} \, = \, \dfrac{1}{4} \, D_{1,+} D_{2,+} \Big( 
u^n_{j-1,k-1} \, u^n_{j,k} \, + \, u^n_{j-1,k} u^n_{j,k-1} \Big) \, ,
$$
and thus obtain the relation:
$$
\sum_{j,k \ge 1} \, D_{1,0} u^n_{j,k} \, D_{2,0} u^n_{j,k} \, + \, u^n_{j,k} \, D_{1,0} D_{2,0} u^n_{j,k} \, = \, 
\dfrac{1}{4} \, (u^n_{0,0} \, u^n_{1,1} \, + \, u^n_{0,1} u^n_{1,0}) \, ,
$$
as well as the relations\footnote{Here we use the boundary conditions \eqref{extrapolation}.}:
\begin{align*}
\sum_{j \ge 1} \, D_{1,0} u^n_{j,0} \, D_{2,0} u^n_{j,0} \, + \, u^n_{j,0} \, D_{1,0} D_{2,0} u^n_{j,0} \, =& 
\, - \, \dfrac{1}{2} \, (u^n_{0,0} \, u^n_{1,1} + u^n_{0,1} \, u^n_{1,0}) \, + \,  u^n_{0,0} \, u^n_{1,0} \, ,\\
\sum_{k \ge 1} \, D_{1,0} u^n_{0,k} \, D_{2,0} u^n_{0,k} \, + \, u^n_{0,k} \, D_{1,0} D_{2,0} u^n_{0,k} \, =& 
\, - \, \dfrac{1}{2} \, (u^n_{0,0} \, u^n_{1,1} + u^n_{0,1} \, u^n_{1,0}) \, + \,  u^n_{0,0} \, u^n_{0,1} \, .
\end{align*}
Adding the interior and boundary contributions together with the corner contribution\footnote{We recall that the boundary conditions 
\eqref{extrapolation} and \eqref{corner} yield $D_{1,0} u^n_{0,0}=D_{1,+} u^n_{0,0}$, $D_{2,0} u^n_{0,0}=D_{2,+} u^n_{0,0}$ and 
$D_{1,0} D_{2,0} u^n_{0,0}=D_{1,+} D_{2,+} u^n_{0,0}$.} at $(0,0)$, we eventually get:
\begin{align*}
\langle D_{1,0} u^n;D_{2,0} u^n \rangle \, + \, \langle u^n;D_{1,0} D_{2,0} u^n \rangle \, =& 
\, - \, \dfrac{1}{4} \, (u^n_{0,0} \, u^n_{1,1} \, + \, u^n_{0,1} u^n_{1,0}) \, + \, \dfrac{1}{2} \, u^n_{0,0} \, (u^n_{0,1} + u^n_{1,0}) \\
&\, + \, \dfrac{1}{4} \, D_{1,+}u_{0,0}^n \, D_{2,+}u_{0,0}^n \, + \, \dfrac{1}{4} \, u_{0,0}^n \, D_{1,+} D_{2,+} u_{0,0}^n \\
=& \, \dfrac{1}{2} \, (u^n_{0,0})^2 \, .
\end{align*}
Recalling that both $\alpha$ and $\beta$ are negative, we end up with:
\begin{equation}
\label{lem2-relation2}
2 \, \alpha \, \beta \, \Big( \langle D_{1,0} u^n;D_{2,0} u^n \rangle \, + \, \langle u^n;D_{1,0} D_{2,0} u^n \rangle \Big) \, = \, 
|\alpha \, \beta| \, (u^n_{0,0})^2 \, ,
\end{equation}
and this gives the first term of the seventh line in \eqref{lem2-sym1}.

It remains to examine the very last term on the right-hand side of \eqref{lem2-expression}. We first recall that $\Delta_1 \Delta_2 u^n_{j,k}$ 
vanishes whenever $j$ or $k$ is zero. This property has already been used in the proof of Lemma \ref{lem1}. We thus have:
$$
\langle u^n;\Delta_1 \Delta_2 u^n \rangle \, = \, \sum_{j,k \ge 1} \, u_{j,k}^n \, \Delta_1 \Delta_2 u^n_{j,k} \, .
$$
We now use twice (alternatively with respect to the second and first variables) the algebraic relation:
\begin{multline}
\label{ipp1d'}
U_\ell \, (V_{\ell+1}-2\, V_\ell+V_{\ell-1}) \, + \, \dfrac{1}{2} \, (U_\ell-U_{\ell-1}) \, (V_\ell-V_{\ell-1}) 
\, + \, \dfrac{1}{2} \, (U_{\ell+1}-U_\ell) \, (V_{\ell+1}-V_\ell) \\
= \, \dfrac{1}{2} \, (U_{\ell+1}+U_\ell) \, (V_{\ell+1}-V_\ell) \, - \, \dfrac{1}{2} \, (U_\ell+U_{\ell-1}) \, (V_\ell-V_{\ell-1}) \, ,
\end{multline}
which yields, using the fact that the right-hand side of \eqref{ipp1d'} is telescopic:
\begin{align}
\langle u^n;\Delta_1 \Delta_2 u^n \rangle \, =& \, \dfrac{1}{4} \, \| D_{1,-} D_{2,-} u^n \|_{\ell^2(\mathring{\mathbb{I}})}^2 
\, + \, \dfrac{1}{4} \, \| D_{1,-} D_{2,+} u^n \|_{\ell^2(\mathring{\mathbb{I}})}^2 
\, + \, \dfrac{1}{4} \, \| D_{1,+} D_{2,-} u^n \|_{\ell^2(\mathring{\mathbb{I}})}^2 \notag \\
&\, + \, \dfrac{1}{4} \, \| D_{1,+} D_{2,+} u^n \|_{\ell^2(\mathring{\mathbb{I}})}^2 
\, - \, \dfrac{1}{2} \, \sum_{j \ge 1} \, (u_{j,0}^n+u_{j,1}^n) \, \Delta_1 D_{2,+} u_{j,0}^n \label{lem2-equation1} \\
&\, + \, \dfrac{1}{4} \, \sum_{k \ge 1} \, (D_{2,-} u_{0,k}^n + D_{2,-} u_{1,k}^n) \, D_{1,+} D_{2,-} u_{0,k}^n \notag\\
&\, + \, \dfrac{1}{4} \, \sum_{k \ge 1} \, (D_{2,+} u_{0,k}^n + D_{2,+} u_{1,k}^n) \, D_{1,+} D_{2,+} u_{0,k}^n \, .\notag
\end{align}
Let us look at the sums with respect to $k$ arising in the last two lines on the right-hand side of \eqref{lem2-equation1}. From the definition 
$D_{1,+} D_{2,-} u_{0,k}^n=D_{2,-} u_{1,k}^n-D_{2,-} u_{0,k}^n$, we get:
\begin{align*}
\sum_{k \ge 1} \, (D_{2,-} u_{0,k}^n + D_{2,-} u_{1,k}^n) \, D_{1,+} D_{2,-} u_{0,k}^n \, =& \, 
\sum_{k \ge 1} \, (D_{2,-} u_{1,k}^n)^2 - (D_{2,-} u_{0,k}^n)^2 \\
=& \, \sum_{k \ge 1} \, (D_{1,+} D_{2,-} u_{0,k}^n)^2 + \, 2 \, D_{2,-} u_{0,k}^n \, D_{1,+} D_{2,-} u_{0,k}^n \\
=& \, \sum_{k \ge 0} \, (D_{1,+} D_{2,+} u_{0,k}^n)^2 + \, 2 \, D_{2,+} u_{0,k}^n \, D_{1,+} D_{2,+} u_{0,k}^n \, .
\end{align*}
The sum in the fourth line of \eqref{lem2-equation1} can be rewritten similarly (except for the very last manipulation which was a shift on the 
index $k$), which yields the following equivalent expression for the scalar product $\langle u^n;\Delta_1 \Delta_2 u^n \rangle$:
\begin{align}
\langle u^n;\Delta_1 \Delta_2 u^n \rangle \, =& \, \dfrac{1}{4} \, \| D_{1,-} D_{2,-} u^n \|_{\ell^2(\mathring{\mathbb{I}})}^2 
\, + \, \dfrac{1}{4} \, \| D_{1,-} D_{2,+} u^n \|_{\ell^2(\mathring{\mathbb{I}})}^2 
\, + \, \dfrac{1}{4} \, \| D_{1,+} D_{2,-} u^n \|_{\ell^2(\mathring{\mathbb{I}})}^2 \notag \\
&\, + \, \dfrac{1}{4} \, \| D_{1,+} D_{2,+} u^n \|_{\ell^2(\mathring{\mathbb{I}})}^2 
\, - \, \dfrac{1}{2} \, \sum_{j \ge 1} \, (u_{j,0}^n+u_{j,1}^n) \, \Delta_1 D_{2,+} u_{j,0}^n \label{lem2-equation2} \\
&\, + \, \dfrac{1}{2} \, \sum_{k \ge 1} \, (D_{1,+} D_{2,+} u_{0,k}^n)^2 \, + \, \sum_{k \ge 1} \, D_{2,+} u_{0,k}^n \, D_{1,+} D_{2,+} u_{0,k}^n \notag\\
&\, + \, \dfrac{1}{4} \, (D_{1,+} D_{2,+} u_{0,0}^n)^2 \, + \, \dfrac{1}{2} \, D_{2,+} u_{0,0}^n \, D_{1,+} D_{2,+} u_{0,0}^n \, .\notag
\end{align}
We now deal with the sum with respect to $j$ in the second line of \eqref{lem2-equation2}. We first decompose:
$$
\sum_{j \ge 1} \, (u_{j,0}^n+u_{j,1}^n) \, \Delta_1 D_{2,+} u_{j,0}^n \, = \, 
2 \, \sum_{j \ge 1} \, u_{j,0}^n \, \Delta_1 D_{2,+} u_{j,0}^n \, + \, \sum_{j \ge 1} \, D_{2,+} u_{j,0}^n \, \Delta_1 D_{2,+} u_{j,0}^n \, ,
$$
and then apply the integration by parts formula \eqref{ipp1d'} to each of the two sums. After a few manipulations, we obtain the expressions:
\begin{align*}
\sum_{j \ge 1} \, u_{j,0}^n \, \Delta_1 D_{2,+} u_{j,0}^n \, =& \, - \, \sum_{j \ge 1} \, D_{1,+} u_{j,0}^n \, D_{1,+} D_{2,+} u_{j,0}^n 
\, - \, (u_{0,0}^n +D_{1,+} u_{0,0}^n) \, D_{1,+} D_{2,+} u_{0,0}^n \\
\sum_{j \ge 1} \, D_{2,+} u_{j,0}^n \, \Delta_1 D_{2,+} u_{j,0}^n \, =& \, - \, \sum_{j \ge 1} \, (D_{1,+} D_{2,+} u_{j,0}^n)^2 
\, - \, (D_{1,+} D_{2,+} u_{0,0}^n)^2  \, - \, D_{2,+} u_{0,0}^n \, D_{1,+} D_{2,+} u_{0,0}^n \, .
\end{align*}
Going back to \eqref{lem2-equation2} and substituting, we obtain:
\begin{align}
\langle u^n;\Delta_1 \Delta_2 u^n \rangle \, =& \, \dfrac{1}{4} \, \| D_{1,-} D_{2,-} u^n \|_{\ell^2(\mathring{\mathbb{I}})}^2 
\, + \, \dfrac{1}{4} \, \| D_{1,-} D_{2,+} u^n \|_{\ell^2(\mathring{\mathbb{I}})}^2 
\, + \, \dfrac{1}{4} \, \| D_{1,+} D_{2,-} u^n \|_{\ell^2(\mathring{\mathbb{I}})}^2 \notag \\
&\, + \, \dfrac{1}{4} \, \| D_{1,+} D_{2,+} u^n \|_{\ell^2(\mathring{\mathbb{I}})}^2 \notag \\
&\, + \, \dfrac{1}{2} \, \sum_{j \ge 1} \, (D_{1,+} D_{2,+} u_{j,0}^n)^2 
\, + \, \dfrac{1}{2} \, \sum_{k \ge 1} \, (D_{1,+} D_{2,+} u_{0,k}^n)^2 \label{lem2-relation3} \\
&\, + \, \sum_{j \ge 1} \, D_{1,+} u_{j,0}^n \, D_{1,+} D_{2,+} u_{j,0}^n \, + \, \sum_{k \ge 1} \, D_{2,+} u_{0,k}^n \, D_{1,+} D_{2,+} u_{0,k}^n \notag \\
&\, + \, (u_{0,0}^n + D_{1,+} u_{0,0}^n + D_{2,+} u_{0,0}^n) \, D_{1,+} D_{2,+} u_{0,0}^n \, + \, \dfrac{3}{4} \, (D_{1,+} D_{2,+} u_{0,0}^n)^2 \, .\notag
\end{align}
We now multiply \eqref{lem2-relation3} by $-(\alpha^2+\beta^2)/4$ and combine with \eqref{lem2-relation1} and \eqref{lem2-relation2} 
to obtain the decomposition \eqref{lem2-sym1}.
\end{proof}

\noindent Eventually we explain how to derive an estimate for the second symmetric term in \eqref{bilan-energie}.

\begin{lemma}
\label{lem3}
Let $a,b<0$. Let $u^n \in \mathbb{H}$, and let the sequence $w^n$ be defined on the set of interior indices $\mathbb{I}$ by \eqref{defvwn'}. 
Then there holds:
\begin{align}
\| \, w^n \, \|^2 \, \le& \, \, 2 \, (\alpha^2+\beta^2) \, \left\{ \dfrac{\alpha^2}{4} \, \| \Delta_1 u^n \|_{\ell^2(\mathring{\mathbb{I}})}^2 
\, + \, \dfrac{\beta^2}{4} \, \| \Delta_2 u^n \|_{\ell^2(\mathring{\mathbb{I}})}^2 
\, + \, \dfrac{\alpha^2+\beta^2}{16} \, \Big( \| D_{1,-} D_{2,-} u^n \|_{\ell^2(\mathring{\mathbb{I}})}^2 \right. \notag \\
& \left. \, + \, \| D_{1,-} D_{2,+} u^n \|_{\ell^2(\mathring{\mathbb{I}})}^2 \, + \, \| D_{1,+} D_{2,-} u^n \|_{\ell^2(\mathring{\mathbb{I}})}^2 
\, + \, \| D_{1,+} D_{2,+} u^n \|_{\ell^2(\mathring{\mathbb{I}})}^2 \Big) \right\} \notag \\
& \, - \, \dfrac{\alpha^2 \, \beta^2}{8} \, \sum_{j \ge 1} (\Delta_1 u_{j,0}^n)^2 
\, - \, \dfrac{\alpha^2 \, \beta^2}{8} \, \sum_{k \ge 1} (\Delta_2 u_{0,k}^n)^2 \label{lem3-sym2} \\
& \, - \, \dfrac{\alpha^2 \, \beta^2}{8} \, \sum_{j \ge 1} (D_{2,+} \Delta_1 u_{j,0}^n)^2 
\, - \, \dfrac{\alpha^2 \, \beta^2}{8} \, \sum_{k \ge 1} (D_{1,+} \Delta_2 u_{0,k}^n)^2 \notag \\
& \, - \, \dfrac{\alpha^2 \, (\alpha^2 + \beta^2)}{8} \, \sum_{j \ge 1} (\Delta_1 u_{j,0}^n) \, D_{2,+} \Delta_1 u_{j,0}^n 
\, - \, \dfrac{\alpha^2 \, (\alpha^2 + \beta^2)}{8} \, \sum_{k \ge 1} (\Delta_2 u_{0,k}^n) \, D_{1,+} \Delta_2 u_{0,k}^n \notag \\
& \, + \, \dfrac{(\alpha^2 + \beta^2)}{8} \, \left\{ \beta^2 \, \sum_{j \ge 1} (D_{1,+} D_{2,+} u_{j,0}^n)^2 
\, + \, \alpha^2 \, \sum_{k \ge 1} (D_{1,+} D_{2,+} u_{0,k}^n)^2 \right\} \notag \\
& \, + \, |\alpha|^3 \, |\beta| \, \sum_{j \ge 1} (\Delta_1 u_{j,0}^n) \, D_{2,+} D_{1,0} u_{j,0}^n 
\, + \,|\alpha| \, |\beta|^3 \, \sum_{k \ge 1} (\Delta_2 u_{0,k}^n) \, D_{1,+} D_{2,0} u_{0,k}^n \notag \\
& \, - \, \dfrac{(\alpha^2 + \beta^2)^2}{16} \, (D_{1,+} D_{2,+} u_{0,0}^n)^2 \, .\notag
\end{align}
\end{lemma}

\begin{proof}
We use the definition \eqref{def:norm} of the norm to get:
$$
4 \, \| \, w^n \, \|^2 \, = \, \sum_{j,k \ge 1} \, (2\, w_{j,k}^n)^2 \, + \, 2 \, \sum_{j \ge 1} \, (w_{j,0}^n)^2 
 \, + \, 2 \, \sum_{k \ge 1} \, (w_{0,k}^n)^2  \, + \, (w_{0,0}^n)^2 \, .
$$
We first make the boundary and corner contributions explicit in Step 1 below. We then derive an estimate for the interior sum (namely the sum 
with respect to both indices $j,k \in \N^*$) in Steps 2 and 3. We conclude in Step 4 by collecting all contributions.
\bigskip

\underline{Step 1 (the boundary and corner contributions).} For $j \ge 1$, the boundary condition \eqref{extrapolationj} gives:
$$
\forall \, j \ge 1 \, ,\quad w_{j,0}^n \, = \, - \, \dfrac{\alpha^2}{2} \, \Delta_1 u_{j,0}^n \, - \, \alpha \, \beta \, D_{2,+} D_{1,0} u_{j,0}^n \, ,
$$
and we thus have:
\begin{align*}
2 \, \sum_{j \ge 1} \, (w_{j,0}^n)^2 \, =& \, \dfrac{\alpha^4}{2} \, \sum_{j \ge 1} \, (\Delta_1 u_{j,0}^n)^2 \, + \, 2 \, \alpha^2 \, \beta^2 \, 
\sum_{j \ge 1} \, (D_{2,+} D_{1,0} u_{j,0}^n)^2 \\
& \, + \, 2 \, \alpha^3 \, \beta \, \sum_{j \ge 1} \, \Delta_1 u_{j,0}^n \, D_{2,+} D_{1,0} u_{j,0}^n \, .
\end{align*}
We then use formula \eqref{formule1d} to obtain:
\begin{align*}
2 \, \sum_{j \ge 1} \, (w_{j,0}^n)^2 \, =& \, \dfrac{\alpha^4}{2} \, \sum_{j \ge 1} \, (\Delta_1 u_{j,0}^n)^2 \, + \, 2 \, \alpha^2 \, \beta^2 \, 
\sum_{j \ge 1} \, (D_{1,+} D_{2,+} u_{j,0}^n)^2 \, - \, \dfrac{\alpha^2 \, \beta^2}{2} \, \sum_{j \ge 1} \, (D_{2,+} \Delta_1 u_{j,0}^n)^2 \\
& \, + \, 2 \, \alpha^3 \, \beta \, \sum_{j \ge 1} \, \Delta_1 u_{j,0}^n \, D_{2,+} D_{1,0} u_{j,0}^n 
\, + \, \alpha^2 \, \beta^2 \, (D_{1,+} D_{2,+} u_{0,0}^n)^2 \, .
\end{align*}
There is, of course, a similar expression for the analogous contribution on the other side of the boundary. Recalling that we have 
$w_{0,0}^n=-\alpha \, \beta \, D_{1,+} D_{2,+} u_{0,0}^n$, we end up with the relation:
\begin{align*}
2 \, \sum_{j \ge 1} \, (w_{j,0}^n)^2 \, + \, 2 \, \sum_{k \ge 1} \, (w_{0,k}^n)^2  \, + \, (w_{0,0}^n)^2 \, =& \, 
\dfrac{\alpha^4}{2} \, \sum_{j \ge 1} \, (\Delta_1 u_{j,0}^n)^2 \, + \, \dfrac{\beta^4}{2} \, \sum_{k \ge 1} \, (\Delta_2 u_{0,k}^n)^2 \\
& \, + \, 2 \, \alpha^2 \, \beta^2 \, \left( \sum_{j \ge 1} \, (D_{1,+} D_{2,+} u_{j,0}^n)^2 
\, + \, \sum_{k \ge 1} \, (D_{1,+} D_{2,+} u_{0,k}^n)^2 \right) \\
& \, - \, \dfrac{\alpha^2 \, \beta^2}{2} \, \left( \sum_{j \ge 1} \, (D_{2,+} \Delta_1 u_{j,0}^n)^2 
\, + \, \sum_{k \ge 1} \, (D_{1,+} \Delta_2 u_{0,k}^n)^2 \right) \\
& \, + \, 2 \, \alpha^3 \, \beta \, \sum_{j \ge 1} \, \Delta_1 u_{j,0}^n \, D_{2,+} D_{1,0} u_{j,0}^n \\
& \, + \, 2 \, \alpha \, \beta^3 \, \sum_{k \ge 1} \, \Delta_2 u_{0,k}^n \, D_{1,+} D_{2,0} u_{0,k}^n \\
& \, + \, 3 \, \alpha^2 \, \beta^2 \, (D_{1,+} D_{2,+} u_{0,0}^n)^2 \, .
\end{align*}
In particular, $\alpha$ and $\beta$ being both negative, the product $\alpha \, \beta$ is positive and we easily obtain our first estimate:
\begin{align}
2 \, \sum_{j \ge 1} \, (w_{j,0}^n)^2 \, + \, 2 \, \sum_{k \ge 1} \, (w_{0,k}^n)^2  \, + \, (w_{0,0}^n)^2 \, \le& \, 
\dfrac{\alpha^4}{2} \, \sum_{j \ge 1} \, (\Delta_1 u_{j,0}^n)^2 \, + \, \dfrac{\beta^4}{2} \, \sum_{k \ge 1} \, (\Delta_2 u_{0,k}^n)^2 \notag \\
& \, + \, \dfrac{(\alpha^2+\beta^2)^2}{2} \, \left( \sum_{j \ge 1} \, (D_{1,+} D_{2,+} u_{j,0}^n)^2 
\, + \, \sum_{k \ge 1} \, (D_{1,+} D_{2,+} u_{0,k}^n)^2 \right) \notag \\
& \, - \, \dfrac{\alpha^2 \, \beta^2}{2} \, \left( \sum_{j \ge 1} \, (D_{2,+} \Delta_1 u_{j,0}^n)^2 
\, + \, \sum_{k \ge 1} \, (D_{1,+} \Delta_2 u_{0,k}^n)^2 \right) \label{lem3-relation1} \\
& \, + \, 2 \, \alpha^3 \, \beta \, \sum_{j \ge 1} \, \Delta_1 u_{j,0}^n \, D_{2,+} D_{1,0} u_{j,0}^n \notag \\
& \, + \, 2 \, \alpha \, \beta^3 \, \sum_{k \ge 1} \, \Delta_2 u_{0,k}^n \, D_{1,+} D_{2,0} u_{0,k}^n \notag \\
& \, + \, \dfrac{3}{2} \, \alpha \, \beta \, (\alpha^2+\beta^2) \, (D_{1,+} D_{2,+} u_{0,0}^n)^2 \, . \notag
\end{align}
We now turn to the interior contribution.
\bigskip

\underline{Step 2 (the interior contribution).} It is convenient to introduce the short-hand notation $\| \cdot \|_{\rm o}$ rather than 
$\| \cdot \|_{\ell^2(\mathring{\mathbb{I}})}$ for the $\ell^2$ norm on the set $\mathring{\mathbb{I}} =\N^* \times \N^*$ (the set of interior 
indices). The corresponding scalar product is denoted $\langle \, ; \, \rangle_{\rm o}$. In other words, we have:
$$
\langle U \, ; \, V \rangle_{\rm o} \, = \, \sum_{j,k \ge 1} \, U_{j,k} \, V_{j,k} \, .
$$
From the definition \eqref{defwjkn'}, we thus compute the expression:
\begin{align*}
4 \, \| \, w^n \, \|_{\rm o}^2 \, =& \, \, \alpha^4 \, \| \Delta_1 u^n \|_{\rm o}^2 \, + \, \beta^4 \, \| \Delta_2 u^n \|_{\rm o}^2 
\, + \, 2 \, \alpha^2 \, \beta^2 \, \langle \Delta_1 u^n \, ; \, \Delta_2 u^n \rangle_{\rm o} \\
& + \, 4 \, \alpha^2 \, \beta^2 \, \| D_{1,0} D_{2,0} u^n \|_{\rm o}^2 \, + \, \dfrac{(\alpha^2+\beta^2)^2}{16} \, \| \Delta_1 \Delta_2 u^n \|_{\rm o}^2 \\
& - \, \dfrac{\alpha^2+\beta^2}{2} \, \langle \Delta_1 \Delta_2 u^n \, ; \, \alpha^2 \, \Delta_1 u^n \, + \, \beta^2 \, \Delta_2 u^n \rangle_{\rm o} \\
& + \, 4 \, \alpha \, \beta \, \langle D_{1,0} \, D_{2,0} u^n \, ; \, \alpha^2 \, \Delta_1 u^n \, + \, \beta^2 \, \Delta_2 u^n \rangle_{\rm o} \\
& - \, (\alpha^2+\beta^2) \, \alpha \, \beta \, \langle D_{1,0} \, D_{2,0} u^n \, ; \, \Delta_1 \Delta_2 u^n \rangle_{\rm o} \, .
\end{align*}
Some of our arguments below are borrowed from our previous work \cite{BC1}. For instance, in the first line on the right-hand side, we use the 
inequality:
$$
2 \, \langle \Delta_1 u^n \, ; \, \Delta_2 u^n \rangle_{\rm o} \, \le \,  \| \Delta_1 u^n \|_{\rm o}^2 \, + \, \| \Delta_2 u^n \|_{\rm o}^2 \, ,
$$
while in the second line of the right-hand side, we use the inequality:
$$
4 \, \alpha^2 \, \beta^2 \, \le \, (\alpha^2+\beta^2)^2 \, ,
$$
which gives:
\begin{align*}
4 \, \| \, w^n \, \|_{\rm o}^2 \, \le & \, \, (\alpha^2 + \beta^2) \, 
\Big( \alpha^2 \, \| \Delta_1 u^n \|_{\rm o}^2 \, + \, \beta^2 \, \| \Delta_2 u^n \|_{\rm o}^2 \Big) 
\, + \, (\alpha^2+\beta^2)^2 \, \left( \| D_{1,0} D_{2,0} u^n \|_{\rm o}^2 \, + \, \dfrac{1}{16} \, \| \Delta_1 \Delta_2 u^n \|_{\rm o}^2 \right) \\
& - \, \dfrac{\alpha^2+\beta^2}{2} \, \langle \Delta_1 \Delta_2 u^n \, ; \, \alpha^2 \, \Delta_1 u^n \, + \, \beta^2 \, \Delta_2 u^n \rangle_{\rm o} 
\, + \, 4 \, \alpha \, \beta \, \langle D_{1,0} \, D_{2,0} u^n \, ; \, \alpha^2 \, \Delta_1 u^n \, + \, \beta^2 \, \Delta_2 u^n \rangle_{\rm o} \\
& - \, (\alpha^2+\beta^2) \, \alpha \, \beta \, \langle D_{1,0} \, D_{2,0} u^n \, ; \, \Delta_1 \Delta_2 u^n \rangle_{\rm o} \, .
\end{align*}

We now use twice the first equality of \eqref{formule1d} to expand the norm $\| D_{1,0} D_{2,0} u^n \|_{\rm o}^2$:
\begin{align*}
\| D_{1,0} D_{2,0} u^n \|_{\rm o}^2 \, =& \, \dfrac{1}{4} \, \Big( 
\| D_{1,-} D_{2,-} u^n \|_{\rm o}^2 + \| D_{1,-} D_{2,+} u^n \|_{\rm o}^2 + \| D_{1,+} D_{2,-} u^n \|_{\rm o}^2 + \| D_{1,+} D_{2,+} u^n \|_{\rm o}^2 \Big) \\
&\, - \, \dfrac{1}{8} \, \Big( \| D_{1,-} \Delta_2 u^n \|_{\rm o}^2 + \| D_{1,+} \Delta_2 u^n \|_{\rm o}^2 + 
\| D_{2,-} \Delta_1 u^n \|_{\rm o}^2 + \| D_{2,+} \Delta_1 u^n \|_{\rm o}^2 \Big) \\
&\, + \, \dfrac{1}{16} \, \| \Delta_1 \Delta_2 u^n \|_{\rm o}^2 \, ,
\end{align*}
and this expression is substituted in the right-hand side of our previous estimate for $4 \, \| \, w^n \, \|_{\rm o}^2$. We thus obtain our first 
preliminary estimate:
\begin{align}
4 \, \| \, w^n \, \|_{\rm o}^2 \, \le& \, \, 
(\alpha^2 + \beta^2) \, \Big\{ \alpha^2 \, \| \Delta_1 u^n \|_{\rm o}^2 \, + \, \beta^2 \, \| \Delta_2 u^n \|_{\rm o}^2 \label{lem3-estim1} \\
& \, + \, \dfrac{(\alpha^2+\beta^2)}{4} \, \Big( \| D_{1,-} D_{2,-} u^n \|_{\rm o}^2 + \| D_{1,-} D_{2,+} u^n \|_{\rm o}^2 
+ \| D_{1,+} D_{2,-} u^n \|_{\rm o}^2 + \| D_{1,+} D_{2,+} u^n \|_{\rm o}^2 \Big) \Big\} \notag \\
& \, + \, \mathcal{A} \, ,\notag
\end{align}
where the quantity $\mathcal{A}$ is defined by:
\begin{align}
\mathcal{A} \, :=& \, \dfrac{(\alpha^2+\beta^2)^2}{8} \, \| \Delta_1 \Delta_2 u^n \|_{\rm o}^2 \notag \\
& \, - \, \dfrac{(\alpha^2+\beta^2)^2}{8} \, \Big( \| D_{1,-} \Delta_2 u^n \|_{\rm o}^2 + \| D_{1,+} \Delta_2 u^n \|_{\rm o}^2 + 
\| D_{2,-} \Delta_1 u^n \|_{\rm o}^2 + \| D_{2,+} \Delta_1 u^n \|_{\rm o}^2 \Big) \label{lem3-defA} \\
& \, - \, \dfrac{\alpha^2+\beta^2}{2} \, \langle \Delta_1 \Delta_2 u^n \, ; \, \alpha^2 \, \Delta_1 u^n \, + \, \beta^2 \, \Delta_2 u^n \rangle_{\rm o} 
\, - \, (\alpha^2+\beta^2) \, \alpha \, \beta \, \langle D_{1,0} \, D_{2,0} u^n \, ; \, \Delta_1 \Delta_2 u^n \rangle_{\rm o} \notag \\
& \, + \, 4 \, \alpha \, \beta \, \langle D_{1,0} \, D_{2,0} u^n \, ; \, \alpha^2 \, \Delta_1 u^n \, + \, \beta^2 \, \Delta_2 u^n \rangle_{\rm o} \, .\notag
\end{align}

We now focus on estimating the quantity $\mathcal{A}$. We use the one-dimensional formula:
$$
\sum_{j \ge 1} \, (\Delta_1 U_j) \, U_j \, = \, - \, \dfrac{1}{2} \, \sum_{j \ge 1} \, (D_{1,-} U_j)^2 \, \, - \, \dfrac{1}{2} \, \sum_{j \ge 1} \, (D_{1,+} U_j)^2 
\, - \, \dfrac{1}{2} \, (U_1-U_0)^2 \, - \, U_0 \, (U_1-U_0) \, ,
$$
and thus integrate by parts the two scalar products $\langle \Delta_1 \Delta_2 u^n \, ; \, \Delta_1 u^n \rangle_{\rm o}$ and 
$\langle \Delta_1 \Delta_2 u^n \, ; \, \Delta_2 u^n \rangle_{\rm o}$. We obtain the equivalent expression:
\begin{align*}
\mathcal{A} \, =& \, \dfrac{(\alpha^2+\beta^2)^2}{8} \, \| \Delta_1 \Delta_2 u^n \|_{\rm o}^2 
\, + \, \dfrac{(\alpha^2+\beta^2)}{8} \, \left\{ (\alpha^2-\beta^2) \, \Big( \| D_{2,-} \Delta_1 u^n \|_{\rm o}^2 + \| D_{2,+} \Delta_1 u^n \|_{\rm o}^2 \Big) \right. \\
& \qquad \qquad  \qquad \qquad \qquad \qquad \qquad \qquad 
\left. \, + \, (\beta^2-\alpha^2) \, \Big( \| D_{1,-} \Delta_2 u^n \|_{\rm o}^2 + \| D_{1,+} \Delta_2 u^n \|_{\rm o}^2 \Big) \right\} \\
&\, - \, (\alpha^2+\beta^2) \, \alpha \, \beta \, \langle D_{1,0} \, D_{2,0} u^n \, ; \, \Delta_1 \Delta_2 u^n \rangle_{\rm o} 
\, + \, 4 \, \alpha \, \beta \, \langle D_{1,0} \, D_{2,0} u^n \, ; \, \alpha^2 \, \Delta_1 u^n \, + \, \beta^2 \, \Delta_2 u^n \rangle_{\rm o} \\
& \, + \, \dfrac{\alpha^2 \, (\alpha^2+\beta^2)}{4} \, \sum_{j \ge 1} \, (D_{2,+} \Delta_1 u^n_{j,0})^2 \, + \, 
\dfrac{\beta^2 \, (\alpha^2+\beta^2)}{4} \, \sum_{k \ge 1} \, (D_{1,+} \Delta_2 u^n_{0,k})^2 \\
& \, + \, \dfrac{\alpha^2 \, (\alpha^2+\beta^2)}{2} \, \sum_{j \ge 1} \, (\Delta_1 u^n_{j,0}) \, D_{2,+} \Delta_1 u^n_{j,0} \, + \, 
\dfrac{\beta^2 \, (\alpha^2+\beta^2)}{2} \, \sum_{k \ge 1} \, (\Delta_2 u^n_{0,k}) \, D_{1,+} \Delta_2 u^n_{0,k} \, .
\end{align*}

We now rewrite the scalar product $\langle D_{1,0} \, D_{2,0} u^n \, ; \, \Delta_1 \Delta_2 u^n \rangle_{\rm o}$ by using the formula:
$$
\langle D_{2,0} \, U \, ; \, \Delta_2 \, V \rangle_{\rm o} \, + \, \langle D_{2,0} \, V \, ; \, \Delta_2 \, U \rangle_{\rm o} 
\, = \, -\sum_{j \ge 1} \, (D_{2,+} U_{j,0}) \, D_{2,+} V_{j,0} \, ,
$$
which gives, after computing the resulting sum with respect to $j$ in yet another telescopic way:
$$
\langle D_{1,0} \, D_{2,0} u^n \, ; \, \Delta_1 \Delta_2 u^n \rangle_{\rm o} \, = \, 
- \, \langle D_{1,0} \, \Delta_2 u^n \, ; \, D_{2,0} \, \Delta_1 u^n \rangle_{\rm o} \, + \, \dfrac{1}{2} \, (D_{1,+} D_{2,+} u^n_{0,0})^2 \, .
$$
We have thus obtained the expression:
\begin{align*}
\mathcal{A} \, =& \, \dfrac{(\alpha^2+\beta^2)^2}{8} \, \| \Delta_1 \Delta_2 u^n \|_{\rm o}^2 
\, + \, \dfrac{(\alpha^2+\beta^2)}{8} \, \left\{ (\alpha^2-\beta^2) \, \Big( \| D_{2,-} \Delta_1 u^n \|_{\rm o}^2 + \| D_{2,+} \Delta_1 u^n \|_{\rm o}^2 \Big) \right. \\
& \qquad \qquad \qquad \qquad \qquad \qquad \qquad \qquad 
\left. \, + \, (\beta^2-\alpha^2) \, \Big( \| D_{1,-} \Delta_2 u^n \|_{\rm o}^2 + \| D_{1,+} \Delta_2 u^n \|_{\rm o}^2 \Big) \right\} \\
&\, + \, (\alpha^2+\beta^2) \, \alpha \, \beta \, \langle D_{1,0} \, \Delta_2 u^n \, ; \, D_{2,0} \, \Delta_1 u^n \rangle_{\rm o} 
\, + \, 4 \, \alpha \, \beta \, \langle D_{1,0} \, D_{2,0} u^n \, ; \, \alpha^2 \, \Delta_1 u^n \, + \, \beta^2 \, \Delta_2 u^n \rangle_{\rm o} \\
& \, + \, \dfrac{\alpha^2 \, (\alpha^2+\beta^2)}{4} \, \sum_{j \ge 1} \, (D_{2,+} \Delta_1 u^n_{j,0})^2 \, + \, 
\dfrac{\beta^2 \, (\alpha^2+\beta^2)}{4} \, \sum_{k \ge 1} \, (D_{1,+} \Delta_2 u^n_{0,k})^2 \\
& \, + \, \dfrac{\alpha^2 \, (\alpha^2+\beta^2)}{2} \, \sum_{j \ge 1} \, (\Delta_1 u^n_{j,0}) \, D_{2,+} \Delta_1 u^n_{j,0} \, + \, 
\dfrac{\beta^2 \, (\alpha^2+\beta^2)}{2} \, \sum_{k \ge 1} \, (\Delta_2 u^n_{0,k}) \, D_{1,+} \Delta_2 u^n_{0,k} \\
& \, - \, \alpha \, \beta \, \dfrac{(\alpha^2+\beta^2)}{2} \, (D_{1,+} D_{2,+} u^n_{0,0})^2 \, .
\end{align*}
Then, as in \cite{BC1}, we focus on the first scalar product in the third line of the right-hand side of the latter expression for $\mathcal{A}$. 
We use Cauchy-Schwarz inequality as well as the inequality: for $a_1,\, a_2,\, a_3,\, a_4\, \in \mathbb{R}$,
\begin{equation}
\label{inegalitea1234}
a_1 \, a_2 \, a_3 \, a_4 \, \le \, \dfrac{1}{4} \, (a_1^2+a_2^2) \, (a_3^2+a_4^2) \, ,
\end{equation}
and we end up with the inequality:
\begin{align*}
\mathcal{A} \, \le& \, \dfrac{(\alpha^2+\beta^2)^2}{8} \, \| \Delta_1 \Delta_2 u^n \|_{\rm o}^2 
\, + \, \dfrac{(\alpha^2+\beta^2)}{8} \, \left\{ (\alpha^2-\beta^2) \, \Big( \| D_{2,-} \Delta_1 u^n \|_{\rm o}^2 + \| D_{2,+} \Delta_1 u^n \|_{\rm o}^2 \Big) \right. \\
& \qquad \qquad \qquad \qquad \qquad \qquad \qquad \qquad 
\left. \, + \, (\beta^2-\alpha^2) \, \Big( \| D_{1,-} \Delta_2 u^n \|_{\rm o}^2 + \| D_{1,+} \Delta_2 u^n \|_{\rm o}^2 \Big) \right\} \\
&\, + \, \dfrac{(\alpha^2+\beta^2)^2}{4} \, \left( \| D_{1,0} \, \Delta_2 u^n \|_{\rm o}^2 \, + \, \| D_{2,0} \, \Delta_1 u^n \|_{\rm o}^2 \right) 
\, + \, 4 \, \alpha \, \beta \, \langle D_{1,0} \, D_{2,0} u^n \, ; \, \alpha^2 \, \Delta_1 u^n \, + \, \beta^2 \, \Delta_2 u^n \rangle_{\rm o} \\
& \, + \, \dfrac{\alpha^2 \, (\alpha^2+\beta^2)}{4} \, \sum_{j \ge 1} \, (D_{2,+} \Delta_1 u^n_{j,0})^2 \, + \, 
\dfrac{\beta^2 \, (\alpha^2+\beta^2)}{4} \, \sum_{k \ge 1} \, (D_{1,+} \Delta_2 u^n_{0,k})^2 \\
& \, + \, \dfrac{\alpha^2 \, (\alpha^2+\beta^2)}{2} \, \sum_{j \ge 1} \, (\Delta_1 u^n_{j,0}) \, D_{2,+} \Delta_1 u^n_{j,0} \, + \, 
\dfrac{\beta^2 \, (\alpha^2+\beta^2)}{2} \, \sum_{k \ge 1} \, (\Delta_2 u^n_{0,k}) \, D_{1,+} \Delta_2 u^n_{0,k} \\
& \, - \, \alpha \, \beta \, \dfrac{(\alpha^2+\beta^2)}{2} \, (D_{1,+} D_{2,+} u^n_{0,0})^2 \, .
\end{align*}
We apply once again the first equality of \eqref{formule1d} and expand the two norms $\| D_{1,0} \, \Delta_2 u^n \|_{\rm o}^2$ and 
$\| D_{2,0} \, \Delta_1 u^n \|_{\rm o}^2$. After simplifying with other terms, we obtain the following estimate for the quantity 
$\mathcal{A}$ defined in \eqref{lem3-defA}:
\begin{align}
\mathcal{A} \, \le& \, \alpha^2 \, \mathcal{A}_1 \, + \, \beta^2 \, \mathcal{A}_2 
\, + \, \dfrac{\alpha^2 \, (\alpha^2+\beta^2)}{4} \, \sum_{j \ge 1} \, (D_{2,+} \Delta_1 u^n_{j,0})^2 \, + \, 
\dfrac{\beta^2 \, (\alpha^2+\beta^2)}{4} \, \sum_{k \ge 1} \, (D_{1,+} \Delta_2 u^n_{0,k})^2 \notag \\
& \, + \, \dfrac{\alpha^2 \, (\alpha^2+\beta^2)}{2} \, \sum_{j \ge 1} \, (\Delta_1 u^n_{j,0}) \, D_{2,+} \Delta_1 u^n_{j,0} \, + \, 
\dfrac{\beta^2 \, (\alpha^2+\beta^2)}{2} \, \sum_{k \ge 1} \, (\Delta_2 u^n_{0,k}) \, D_{1,+} \Delta_2 u^n_{0,k} \label{lem3-expressionA} \\
& \, - \, \alpha \, \beta \, \dfrac{(\alpha^2+\beta^2)}{2} \, (D_{1,+} D_{2,+} u^n_{0,0})^2 \, ,\notag
\end{align}
where the expression of the terms $\mathcal{A}_1$ and $\mathcal{A}_2$ is the following (compare with \cite{BC1} where the analogous terms are 
denoted $B_1$ and $B_2$):
\begin{subequations}
\label{lem3-defA12}
\begin{align}
\mathcal{A}_1 \, :=& \, \dfrac{(\alpha^2+\beta^2)}{4} \, \Big( \| D_{2,-} \Delta_1 u^n \|_{\rm o}^2 + \| D_{2,+} \Delta_1 u^n \|_{\rm o}^2 \Big) 
\, + \, 4 \, \alpha \, \beta \, \langle D_{1,0} D_{2,0} u^n \, ; \, \Delta_1 u^n \rangle_{\rm o} \, , \label{lem3-defA1} \\
\mathcal{A}_2 \, :=& \, \dfrac{(\alpha^2+\beta^2)}{4} \, \Big( \| D_{1,-} \Delta_2 u^n \|_{\rm o}^2 + \| D_{1,+} \Delta_2 u^n \|_{\rm o}^2 \Big) 
\, + \, 4 \, \alpha \, \beta \, \langle D_{1,0} D_{2,0} u^n \, ; \, \Delta_2 u^n \rangle_{\rm o} \, . \label{lem3-defA2}
\end{align}
\end{subequations}
The third step of the proof is to estimate both terms $\mathcal{A}_1$ and $\mathcal{A}_2$.
\bigskip

\underline{Step 3 (the interior contribution).} Following \cite{BC1}, we introduce the averaging operators $\mathbb{A}_1$ and $\mathbb{A}_2$ 
defined by:
\begin{align*}
(\mathbb{A}_{1,+} V)_{j,k} \, &:= \, \dfrac{V_{j,k}+V_{j+1,k}}{2} \, ,\quad (\mathbb{A}_{2,+} V)_{j,k} \, := \, \dfrac{V_{j,k}+V_{j,k+1}}{2} \, ,\\
(\mathbb{A}_{1,-} V)_{j,k} \, &:= \, \dfrac{V_{j-1,k}+V_{j,k}}{2} \, ,\quad (\mathbb{A}_{2,-} V)_{j,k} \, := \, \dfrac{V_{j,k-1}+V_{j,k}}{2} \, ,
\end{align*}
that verify, for instance, $D_{2,0}=D_{2,+} \, \mathbb{A}_{2,-}$. We then use the following formula:
\begin{equation*}
(D_{2,0}V_{j,k}) \, W_{j,k} \, - \, (D_{2,+}V_{j,k}) \, (\mathbb{A}_{2,+} W_{j,k}) \, = \, 
\dfrac{1}{2} \, (D_{2,-}V_{j,k}) \, W_{j,k} \, - \, \dfrac{1}{2} \, (D_{2,+}V_{j,k}) \, W_{j,k+1} \, , 
\end{equation*}
where the right-hand side is telescopic with respect to $k$. We thus obtain the equivalent expression:
\begin{align*}
\langle D_{1,0} D_{2,0} u^n \, ; \, \Delta_1 u^n \rangle_{\rm o} \, =& \, 
\langle D_{2,+} D_{1,0} u^n \, ; \, \mathbb{A}_{2,+} \Delta_1 u^n \rangle_{\rm o} 
\, + \, \dfrac{1}{2} \, \sum_{j \ge 1} \, (\Delta_1 u^n_{j,1}) \, D_{2,+} D_{1,0} u^n_{j,0} \\
=& \, \langle D_{2,+} D_{1,0} u^n \, ; \, \mathbb{A}_{2,+} \Delta_1 u^n \rangle_{\rm o} 
\, + \, \dfrac{1}{2} \, \sum_{j \ge 1} \, (\Delta_1 u^n_{j,0} ) \, D_{2,+} D_{1,0} u^n_{j,0} \, - \, \dfrac{1}{4} \, (D_{1,+} D_{2,+} u^n_{0,0})^2 \, ,
\end{align*}
where we write $u_{j,1}^n\,=\,D_{2,+}\,u_{j,0}^n\,+\,u_{j,0}^n$ in the first sum on the right-hand side. Starting from the definition \eqref{lem3-defA1}, 
we thus obtain the expression:
\begin{multline*}
\mathcal{A}_1 \, = \, \dfrac{(\alpha^2+\beta^2)}{4} \, \Big( \| D_{2,-} \Delta_1 u^n \|_{\rm o}^2 + \| D_{2,+} \Delta_1 u^n \|_{\rm o}^2 \Big) 
\, + \, 4 \, \alpha \, \beta \, \langle D_{2,+} D_{1,0} u^n \, ; \, \mathbb{A}_{2,+} \Delta_1 u^n \rangle_{\rm o} \\
\, + \, 2 \, \alpha \, \beta \, \sum_{j \ge 1} \, (\Delta_1 u^n_{j,0} ) \, D_{2,+} D_{1,0} u^n_{j,0} 
\, - \, \alpha \, \beta \, (D_{1,+} D_{2,+} u^n_{0,0})^2 \, .
\end{multline*}
We then apply Cauchy-Schwarz inequality to the scalar product in the first line and we use again inequality \eqref{inegalitea1234} to get:
\begin{multline*}
\mathcal{A}_1 \, \le \, \dfrac{(\alpha^2+\beta^2)}{4} \, \Big( \| D_{2,-} \Delta_1 u^n \|_{\rm o}^2 + \| D_{2,+} \Delta_1 u^n \|_{\rm o}^2 
+ 4 \, \| D_{2,+} D_{1,0} u^n \|_{\rm o}^2 + 4 \, \| \mathbb{A}_{2,+} \Delta_1 u^n \|_{\rm o}^2 \Big) \\
\, + \, 2 \, \alpha \, \beta \, \sum_{j \ge 1} \, (\Delta_1 u^n_{j,0} ) \, D_{2,+} D_{1,0} u^n_{j,0} 
\, - \, \alpha \, \beta \, (D_{1,+} D_{2,+} u^n_{0,0})^2 \, .
\end{multline*}
By shifting indices in the norm $\| D_{2,-} \Delta_1 u^n \|_{\rm o}^2$, we therefore have:
\begin{multline}
\label{premiereestimationA1}
\mathcal{A}_1 \, \le \, \dfrac{(\alpha^2+\beta^2)}{4} \, \Big( 2 \, \| D_{2,+} \Delta_1 u^n \|_{\rm o}^2 + 4 \, \| D_{2,+} D_{1,0} u^n \|_{\rm o}^2 
+ 4 \, \| \mathbb{A}_{2,+} \Delta_1 u^n \|_{\rm o}^2 \Big) \\
\, + \, \dfrac{(\alpha^2+\beta^2)}{4} \, \sum_{j \ge 1} \, (D_{2,+} \Delta_1 u^n_{j,0})^2 
\, + \, 2 \, \alpha \, \beta \, \sum_{j \ge 1} \, (\Delta_1 u^n_{j,0}) \, D_{2,+} D_{1,0} u^n_{j,0} 
\, - \, \alpha \, \beta \, (D_{1,+} D_{2,+} u^n_{0,0})^2 \, .
\end{multline}
We then expand the two norms $\| D_{2,+} D_{1,0} u^n \|_{\rm o}$ and $\| \mathbb{A}_{2,+} \Delta_1 u^n \|_{\rm o}$ as follows. We 
use again the first identity of \eqref{formule1d} for $\| D_{2,+} D_{1,0} u^n \|_{\rm o}^2$ and a straightforward computation for the norm 
$\| \mathbb{A}_{2,+} \Delta_1 u^n \|_{\rm o}^2$:
\begin{align*}
4 \, \| D_{2,+} D_{1,0} u^n \|_{\rm o}^2 \, &= \, 2 \, \| D_{1,-} D_{2,+} u^n \|_{\rm o}^2 \, + \, 2 \, \| D_{1,+} D_{2,+} u^n \|_{\rm o}^2 
\, - \, \| D_{2,+} \Delta_1 u^n \|_{\rm o}^2 \, ,\\
4 \, \| \mathbb{A}_{2,+} \Delta_1 u^n \|_{\rm o}^2 \, &= \, 4 \, \| \Delta_1 u^n \|_{\rm o}^2 \, - \, \| D_{2,+} \Delta_1 u^n \|_{\rm o}^2 
\, - \, 2 \, \sum_{j \ge 1} \, (\Delta_1 u^n_{j,1})^2 \, .
\end{align*}
By substituting the previous two relations in the right-hand side of \eqref{premiereestimationA1}, and by further expanding $\Delta_1 u^n_{j,1} 
=\Delta_1 u^n_{j,0} +D_{2,+}\Delta_1 u^n_{j,0}$, we obtain the estimate:
\begin{align*}
\mathcal{A}_1 \, \le& \, \dfrac{(\alpha^2+\beta^2)}{2} \, \Big( \| D_{1,-} D_{2,+} u^n \|_{\rm o}^2 \, + \, \| D_{1,+} D_{2,+} u^n \|_{\rm o}^2 
\, + \, 2 \, \| \Delta_1 u^n \|_{\rm o}^2 \Big) \, - \, \dfrac{(\alpha^2+\beta^2)}{2} \, \sum_{j \ge 1} \, (\Delta_1 u^n_{j,0})^2 \\
&- \, \dfrac{(\alpha^2+\beta^2)}{4} \, \sum_{j \ge 1} \, (D_{2,+} \Delta_1 u^n_{j,0})^2 
\, - \, (\alpha^2+\beta^2) \, \sum_{j \ge 1} \, (\Delta_1 u^n_{j,0}) \, D_{2,+} \Delta_1 u^n_{j,0} \\
&+ \, 2 \, \alpha \, \beta \, \sum_{j \ge 1} \, (\Delta_1 u^n_{j,0}) \, D_{2,+} D_{1,0} u^n_{j,0} 
\, - \, \alpha \, \beta \, (D_{1,+} D_{2,+} u^n_{0,0})^2 \, .
\end{align*}
Combining with the analogous estimate for $\mathcal{A}_2$, we end up after a few simplifications with:
\begin{align}
\alpha^2 \, \mathcal{A}_1 \, +& \, \beta^2 \, \mathcal{A}_2 \notag \\
\le& \, (\alpha^2+\beta^2) \, \Big( \alpha^2 \, \| \Delta_1 u^n \|_{\rm o}^2 \, + \, \beta^2 \, \| \Delta_2 u^n \|_{\rm o}^2 \Big) \notag \\
&+ \, \dfrac{(\alpha^2+\beta^2)^2}{4} \, \Big( \| D_{1,-} D_{2,-} u^n \|_{\rm o}^2 \, + \, \| D_{1,-} D_{2,+} u^n \|_{\rm o}^2 
\, + \, \| D_{1,+} D_{2,-} u^n \|_{\rm o}^2 \, + \, \| D_{1,+} D_{2,+} u^n \|_{\rm o}^2 \Big) \notag \\
&- \, \dfrac{\alpha^2 \, (\alpha^2+\beta^2)}{4} \, \sum_{j \ge 1} \, (D_{2,+} \Delta_1 u^n_{j,0})^2 
\, - \, \dfrac{\beta^2 \, (\alpha^2+\beta^2)}{4} \, \sum_{k \ge 1} \, (D_{1,+} \Delta_2 u^n_{0,k})^2 \notag \\
&- \, \alpha^2 \, (\alpha^2+\beta^2) \, \sum_{j \ge 1} \, (\Delta_1 u^n_{j,0}) \, D_{2,+} \Delta_1 u^n_{j,0} 
\, - \, \beta^2 \, (\alpha^2+\beta^2) \, \sum_{k \ge 1} \, (\Delta_2 u^n_{0,k}) \, D_{1,+} \Delta_2 u^n_{0,k} \label{estim-lem3-intermediaire} \\
&- \, \dfrac{\alpha^2 \, (\alpha^2+\beta^2)}{2} \, \sum_{j \ge 1} \, (\Delta_1 u^n_{j,0})^2 
\, - \, \dfrac{\beta^2 \, (\alpha^2+\beta^2)}{2} \, \sum_{k \ge 1} \, (\Delta_2 u^n_{0,k})^2 \notag \\
&- \, \dfrac{\alpha^2 \, (\alpha^2+\beta^2)}{2} \, \sum_{j \ge 1} \, (D_{1,+} D_{2,+} u^n_{j,0})^2 
\, - \, \dfrac{\beta^2 \, (\alpha^2+\beta^2)}{2} \, \sum_{k \ge 1} \, (D_{1,+} D_{2,+} u^n_{0,k})^2 \notag \\
&+ \, 2 \, \alpha^3 \, \beta \, \sum_{j \ge 1} \, (\Delta_1 u^n_{j,0}) \, D_{2,+} D_{1,0} u^n_{j,0} 
\, + \, 2 \, \alpha \, \beta^3 \, \sum_{k \ge 1} \, (\Delta_2 u^n_{0,k}) \, D_{1,+} D_{2,0} u^n_{0,k} \notag \\
&- \, \alpha \, \beta \, (\alpha^2+\beta^2) \, (D_{1,+} D_{2,+} u^n_{0,0})^2 
\, - \, \dfrac{(\alpha^2+\beta^2)^2}{4} \, (D_{1,+} D_{2,+} u^n_{0,0})^2 \, .\notag
\end{align}
\bigskip

\underline{Step 4 (conclusion).} We first use the estimate \eqref{estim-lem3-intermediaire} in the estimate \eqref{lem3-expressionA} for 
the quantity $\mathcal{A}$ that is defined in \eqref{lem3-defA}. The quantity $\mathcal{A}$ arises in the right-hand side of the estimate 
\eqref{lem3-estim1}. We obtain in this way an estimate for the interior norm $4 \, \| \, w^n \, \|_{\rm o}^2$. This estimate reads:
\begin{align*}
4 \, \| \, w^n \, \|_{\rm o}^2 \, \le& \, 2 \, (\alpha^2+\beta^2) \, 
\Big( \alpha^2 \, \| \Delta_1 u^n \|_{\rm o}^2 \, + \, \beta^2 \, \| \Delta_2 u^n \|_{\rm o}^2 \Big) \\
&+ \, \dfrac{(\alpha^2+\beta^2)^2}{2} \, \Big( \| D_{1,-} D_{2,-} u^n \|_{\rm o}^2 \, + \, \| D_{1,-} D_{2,+} u^n \|_{\rm o}^2 
\, + \, \| D_{1,+} D_{2,-} u^n \|_{\rm o}^2 \, + \, \| D_{1,+} D_{2,+} u^n \|_{\rm o}^2 \Big) \\
&- \, \dfrac{\alpha^2 \, (\alpha^2+\beta^2)}{2} \, \sum_{j \ge 1} \, (\Delta_1 u^n_{j,0}) \, D_{2,+} \Delta_1 u^n_{j,0} 
\, - \, \dfrac{\beta^2 \, (\alpha^2+\beta^2)}{2} \, \sum_{k \ge 1} \, (\Delta_2 u^n_{0,k}) \, D_{1,+} \Delta_2 u^n_{0,k} \\
&- \, \dfrac{\alpha^2 \, (\alpha^2+\beta^2)}{2} \, \sum_{j \ge 1} \, (\Delta_1 u^n_{j,0})^2 
\, - \, \dfrac{\beta^2 \, (\alpha^2+\beta^2)}{2} \, \sum_{k \ge 1} \, (\Delta_2 u^n_{0,k})^2 \\
&- \, \dfrac{\alpha^2 \, (\alpha^2+\beta^2)}{2} \, \sum_{j \ge 1} \, (D_{1,+} D_{2,+} u^n_{j,0})^2 
\, - \, \dfrac{\beta^2 \, (\alpha^2+\beta^2)}{2} \, \sum_{k \ge 1} \, (D_{1,+} D_{2,+} u^n_{0,k})^2 \\
&+ \, 2 \, \alpha^3 \, \beta \, \sum_{j \ge 1} \, (\Delta_1 u^n_{j,0}) \, D_{2,+} D_{1,0} u^n_{j,0} 
\, + \, 2 \, \alpha \, \beta^3 \, \sum_{k \ge 1} \, (\Delta_2 u^n_{0,k}) \, D_{1,+} D_{2,0} u^n_{0,k} \\
&- \, \dfrac{3}{2} \, \alpha \, \beta \, (\alpha^2+\beta^2) \, (D_{1,+} D_{2,+} u^n_{0,0})^2 
\, - \, \dfrac{(\alpha^2+\beta^2)^2}{4} \, (D_{1,+} D_{2,+} u^n_{0,0})^2 \, .
\end{align*}
We then combine this estimate of the interior norm with the estimate \eqref{lem3-relation1} for the boundary and corner terms, and we 
therefore obtain the estimate \eqref{lem3-sym2} of Lemma \ref{lem3} (recalling that both $\alpha$ and $\beta$ are negative).
\end{proof}

\subsection{Proof of the main result}

We go back to the decomposition \eqref{bilan-energie} and then combine Lemma \ref{lem1}, Lemma \ref{lem2} and Lemma \ref{lem3}. We 
obtain the energy inequality:
\begin{equation}
\label{decomp-final}
\| \, u^{n+1} \, \|^2 \, - \, \| \, u^n \, \|^2 \, \le \, \mathcal{I} \, + \, \mathcal{B}_1  \, + \, \mathcal{B}_2  \, + \, \mathcal{C} \, , 
\end{equation}
where $\mathcal{I}$ incorporates the interior contributions, namely:
\begin{multline}
\label{defbilaninterieur}
\mathcal{I} \, := \, \Big( - \, 1 \, + \, 2 \, (\alpha^2+\beta^2) \Big) \, \left\{ \dfrac{\alpha^2}{4} \, \| \Delta_1 u^n \|_o^2 
\, + \, \dfrac{\beta^2}{4} \, \| \Delta_2 u^n \|_o^2 \, + \, \dfrac{\alpha^2+\beta^2}{16} \, \Big( \| D_{1,-} D_{2,-} u^n \|_o^2 \right. \\
\left. \, + \, \| D_{1,-} D_{2,+} u^n \|_o^2 \, + \, \| D_{1,+} D_{2,-} u^n \|_o^2 \, + \, \| D_{1,+} D_{2,+} u^n \|_o^2 \Big) \right\} \, ,
\end{multline}
the term $\mathcal{B}_1$, resp. $\mathcal{B}_2$, incorporates all the contributions on the boundary $\{ k=0, j \ge 1 \}$, resp. $\{ j=0, k \ge 1 \}$, 
namely\footnote{We only give the definition for the term $\mathcal{B}_1$ and leave the analogous definition for $\mathcal{B}_2$ to the interested 
reader.}:
\begin{align}
\mathcal{B}_1 \, :=& \, - \, |\beta| \, \sum_{j \ge 1} \, (u_{j,0}^n)^2 \, - \, \dfrac{|\beta|^3}{2} \, \sum_{j \ge 1} (D_{2,+} u_{j,0}^n)^2 
\, - \, \dfrac{\alpha^2 \, |\beta|}{2} \, \sum_{j \ge 1} (D_{1,+} u_{j,0}^n)^2 
\, - \, \dfrac{\alpha^2 \, (1-|\beta|)^2}{8} \, \sum_{j \ge 1} (\Delta_1 u_{j,0}^n)^2 \notag \\
& \, - \, \beta^2 \, \sum_{j \ge 1} u_{j,0}^n \, D_{2,+} u_{j,0}^n \, - \, |\alpha| \, \beta^2 \, \sum_{j \ge 1} D_{1,0} u_{j,0}^n \, D_{2,+} u_{j,0}^n 
\, - \, \dfrac{\alpha^2 \, \beta^2}{8} \, \sum_{j \ge 1} (D_{2,+} \Delta_1 u_{j,0}^n)^2 \notag \\
& \, + \, |\alpha| \, \dfrac{(\alpha^2+\beta^2)}{4} \, \sum_{j \ge 1} D_{1,0} u_{j,0}^n \, D_{2,+} \Delta_1 u_{j,0}^n 
\, - \, \dfrac{(1+|\beta|-\beta^2) \, (\alpha^2+\beta^2)}{8} \, \sum_{j \ge 1} (D_{1,+} D_{2,+} u_{j,0}^n)^2 \label{defbilanbord} \\
& \, - \, \dfrac{(\alpha^2+\beta^2)}{4} \, \sum_{j \ge 1} D_{1,+} u_{j,0}^n \, D_{1,+} D_{2,+} u_{j,0}^n 
\, - \, \dfrac{\alpha^2 \, (\alpha^2 + \beta^2)}{8} \, \sum_{j \ge 1} (\Delta_1 u_{j,0}^n) \, D_{2,+} \Delta_1 u_{j,0}^n \notag \\
& \, + \, |\alpha|^3 \, |\beta| \, \sum_{j \ge 1} (\Delta_1 u_{j,0}^n) \, D_{2,+} D_{1,0} u_{j,0}^n \, ,\notag
\end{align}
and the term $\mathcal{C}$ incorporates the corner contributions, namely:
\begin{align}
\mathcal{C} \, :=& \, \left( |\alpha| \, |\beta| \, - \, \dfrac{|\alpha|+|\beta|}{2} \right) \, (u_{0,0}^n)^2 \notag \\
& - \, \left( \dfrac{|\alpha|^3}{4} \, + \, \dfrac{\alpha^2 \, |\beta|}{2} \right) \, (D_{1,+} u_{0,0}^n)^2 
\, - \, \left( \dfrac{|\beta|^3}{4} \, + \, \dfrac{|\alpha| \, \beta^2}{2} \right) \, (D_{2,+} u_{0,0}^n)^2 \notag \\
& \, - \, \dfrac{\alpha^2}{2} \, u_{0,0}^n \, D_{1,+} u_{0,0}^n \, - \, \dfrac{\beta^2}{2} \, u_{0,0}^n \, D_{2,+} u_{0,0}^n 
\, - \, \dfrac{|\alpha \, \beta|}{2} \, (|\alpha|+|\beta|) \, D_{1,+} u_{0,0}^n \, D_{2,+} u_{0,0}^n \label{defbilancoin} \\
& \, - \, \dfrac{(\alpha^2+\beta^2)}{4} \, \Big( u_{0,0}^n + D_{1,+} u_{0,0}^n + D_{2,+} u_{0,0}^n \Big) \, D_{1,+} D_{2,+} u_{0,0}^n 
\, - \, \dfrac{3 \, (\alpha^2+\beta^2)}{16} \, (D_{1,+} D_{2,+} u_{0,0}^n)^2 \, ,\notag \\
&\, - \, (|\alpha|+|\beta|) \, \dfrac{(\alpha^2+\beta^2)}{8} \, (D_{1,+} D_{2,+} u_{0,0}^n)^2 
\, - \, \dfrac{(\alpha^2 + \beta^2)^2}{16} \, (D_{1,+} D_{2,+} u_{0,0}^n)^2 \, .\notag
\end{align}
\bigskip

Let us split the analysis below in three steps, which correspond to the interior, corner and boundary contributions. The ordering corresponds 
to an increasing level of difficulty. There is in the end an easy concluding argument.
\bigskip

\underline{Step 1 (the interior contribution).} We first deal with the term $\mathcal{I}$ defined in \eqref{defbilaninterieur}. We first choose 
the parameter $\varepsilon:=1/4$ and assume that $(\alpha,\beta)=(\lambda \, a,\mu \, b)$ satisfy:
\begin{equation}
\label{restrictionCFL}
\alpha^2 \, + \, \beta^2 \, \le \, \varepsilon \, .
\end{equation}
We see from the defining equation \eqref{defbilaninterieur} that we have:
\begin{equation}
\label{estim-bilan-int}
\mathcal{I} \, \le \, - \,  \dfrac{\alpha^2}{8} \, \| \Delta_1 u^n \|_o^2 \, - \, \dfrac{\beta^2}{8} \, \| \Delta_2 u^n \|_o^2 \, ,
\end{equation}
where we have not kept all non-positive contributions on the right-hand side of \eqref{defbilaninterieur} but only the two most simple ones. 
It would be possible to keep more contributions but, in our opinion, that would not change significantly the main result of this article since 
the main feature of the dissipation estimate for the Lax-Wendroff scheme is a fourth order dissipation with respect to both spatial directions.

We shall allow ourselves in what follows to further decrease the value of $\varepsilon$ and will always assume that \eqref{restrictionCFL} 
holds.
\bigskip

\underline{Step 2 (the corner contribution).} We consider the quantity $\mathcal{C}$ defined in \eqref{defbilancoin}. We first try to absorb 
some of the cross terms. We use Young's inequality to estimate the fourth line in the right-hand side of \eqref{defbilancoin}:
\begin{multline*}
\dfrac{(\alpha^2+\beta^2)}{4} \, \Big| u_{0,0}^n + D_{1,+} u_{0,0}^n + D_{2,+} u_{0,0}^n \Big| \, \Big| D_{1,+} D_{2,+} u_{0,0}^n \Big| \\
\le \, \dfrac{(\alpha^2+\beta^2)}{4} \, \Big( (u_{0,0}^n)^2 \, + \, (D_{1,+} u_{0,0}^n)^2 \, + \, (D_{2,+} u_{0,0}^n)^2 \Big) 
\, + \, \dfrac{3 \, (\alpha^2+\beta^2)}{16} \, (D_{1,+} D_{2,+} u_{0,0}^n)^2 \, .
\end{multline*}
Estimating $|\alpha| \, |\beta| \le (\alpha^2+\beta^2)/2$ in the coefficient of $(u_{0,0}^n)^2$ in the first line of \eqref{defbilancoin}, we thus obtain 
the first estimate:
\begin{align}
\mathcal{C} \, \le & \, \left( \dfrac{3 \, (\alpha^2+\beta^2)}{4} \, - \, \dfrac{|\alpha|+|\beta|}{2} \right) \, (u_{0,0}^n)^2 \notag \\
& + \, \left( \dfrac{(\alpha^2+\beta^2)}{4} \, - \, \dfrac{|\alpha|^3}{4} \, - \, \dfrac{\alpha^2 \, |\beta|}{2} \right) \, (D_{1,+} u_{0,0}^n)^2 
\, + \, \left( \dfrac{(\alpha^2+\beta^2)}{4} \, - \, \dfrac{|\beta|^3}{4} \, - \, \dfrac{|\alpha| \, \beta^2}{2} \right) \, (D_{2,+} u_{0,0}^n)^2 \notag \\
& \, - \, \dfrac{\alpha^2}{2} \, u_{0,0}^n \, D_{1,+} u_{0,0}^n \, - \, \dfrac{\beta^2}{2} \, u_{0,0}^n \, D_{2,+} u_{0,0}^n 
\, - \, \dfrac{|\alpha \, \beta|}{2} \, (|\alpha|+|\beta|) \, D_{1,+} u_{0,0}^n \, D_{2,+} u_{0,0}^n \label{bilancoin1} \\
&\, - \, (|\alpha|+|\beta|) \, \dfrac{(\alpha^2+\beta^2)}{8} \, (D_{1,+} D_{2,+} u_{0,0}^n)^2 
\, - \, \dfrac{(\alpha^2 + \beta^2)^2}{16} \, (D_{1,+} D_{2,+} u_{0,0}^n)^2 \, .\notag
\end{align}
We keep on estimating some cross terms and now deal with the product $D_{1,+} u_{0,0}^n \, D_{2,+} u_{0,0}^n$ in the third line of the 
right-hand side of \eqref{bilancoin1}. We use again Young's inequality to obtain:
\begin{align*}
\dfrac{\alpha^2 \, |\beta|}{2} \, \left| D_{1,+} u_{0,0}^n \right| \, \left| D_{2,+} u_{0,0}^n \right| \, &\le \, 
\dfrac{|\alpha|^3}{4} \, (D_{1,+} u_{0,0}^n)^2 \, + \, \dfrac{|\alpha| \, \beta^2}{4} \, (D_{2,+} u_{0,0}^n)^2 \, , \\
\dfrac{|\alpha| \, \beta^2}{2} \, \left| D_{1,+} u_{0,0}^n \right| \, \left| D_{2,+} u_{0,0}^n \right| \, &\le \, 
\dfrac{\alpha^2 \, |\beta|}{4} \, (D_{1,+} u_{0,0}^n)^2 \, + \, \dfrac{|\beta|^3}{4} \, (D_{2,+} u_{0,0}^n)^2 \, .
\end{align*}
This gives our second estimate for the corner contribution $\mathcal{C}$:
\begin{align}
\mathcal{C} \, \le & \, \left( \dfrac{3 \, (\alpha^2+\beta^2)}{4} \, - \, \dfrac{|\alpha|+|\beta|}{2} \right) \, (u_{0,0}^n)^2 
\, - \, \dfrac{\alpha^2}{2} \, u_{0,0}^n \, D_{1,+} u_{0,0}^n \, - \, \dfrac{\beta^2}{2} \, u_{0,0}^n \, D_{2,+} u_{0,0}^n \notag \\
& + \, \left( \dfrac{(\alpha^2+\beta^2)}{4} \, - \, \dfrac{\alpha^2 \, |\beta|}{4} \right) \, (D_{1,+} u_{0,0}^n)^2 
\, + \, \left( \dfrac{(\alpha^2+\beta^2)}{4} \, - \, \dfrac{|\alpha| \, \beta^2}{4} \right) \, (D_{2,+} u_{0,0}^n)^2 \label{bilancoin2} \\
&\, - \, (|\alpha|+|\beta|) \, \dfrac{(\alpha^2+\beta^2)}{8} \, (D_{1,+} D_{2,+} u_{0,0}^n)^2 
\, - \, \dfrac{(\alpha^2 + \beta^2)^2}{16} \, (D_{1,+} D_{2,+} u_{0,0}^n)^2 \, .\notag
\end{align}
The very last term on the right-hand side of \eqref{bilancoin2} has the good sign but will not help in the analysis below. We thus feel 
free to discard this last term.

Unlike what we did in \cite{BC1}, we estimate the cross terms $u_{0,0}^n \, D_{1,+} u_{0,0}^n$ and $u_{0,0}^n \, D_{2,+} u_{0,0}^n$ 
rather crudely since there are already ``bad'' terms of the form $(u_{0,0}^n)^2$ but more importantly there are ``bad'' terms of the form 
$(D_{1,+} u_{0,0}^n)^2$ and $(D_{2,+} u_{0,0}^n)^2$ that cannot be absorbed, even by taking $(\alpha,\beta)$ small enough. We have:
\begin{align*}
\dfrac{\alpha^2}{2} \, \left| u_{0,0}^n \right| \, \left| D_{1,+} u_{0,0}^n \right| \, &\le \, 
\dfrac{\alpha^2}{4} \, (u_{0,0}^n)^2 \, + \, \dfrac{\alpha^2}{4} \, (D_{1,+} u_{0,0}^n)^2 \, , \\
\dfrac{\beta^2}{2} \, \left| u_{0,0}^n \right| \, \left| D_{2,+} u_{0,0}^n \right| \, &\le \, 
\dfrac{\beta^2}{4} \, (u_{0,0}^n)^2 \, + \, \dfrac{\beta^2}{4} \, (D_{2,+} u_{0,0}^n)^2 \, ,
\end{align*}
and we thus get from \eqref{bilancoin2} the estimate:
\begin{align*}
\mathcal{C} \, \le & \, \left( \alpha^2+\beta^2 \, - \, \dfrac{|\alpha|+|\beta|}{2} \right) \, (u_{0,0}^n)^2 
\, - \, \dfrac{\alpha^2 \, |\beta|}{4} \, (D_{1,+} u_{0,0}^n)^2  \, - \, \dfrac{|\alpha| \, \beta^2}{4} \, (D_{2,+} u_{0,0}^n)^2 \\
& + \, \dfrac{(2 \, \alpha^2+\beta^2)}{4} \, (D_{1,+} u_{0,0}^n)^2 \, + \, \dfrac{(\alpha^2+2 \, \beta^2)}{4} \, (D_{2,+} u_{0,0}^n)^2 
\, - \, (|\alpha|+|\beta|) \, \dfrac{(\alpha^2+\beta^2)}{8} \, (D_{1,+} D_{2,+} u_{0,0}^n)^2 \, .
\end{align*}
Up to further restricting the parameters $(\alpha,\beta)$ by choosing a smaller value for $\varepsilon$, the corner contribution $\mathcal{C}$ 
satisfies the estimate:
\begin{align}
\mathcal{C} \, \le & \, - \dfrac{|\alpha|+|\beta|}{4} \, (u_{0,0}^n)^2 
\, - \, \dfrac{\alpha^2 \, |\beta|}{4} \, (D_{1,+} u_{0,0}^n)^2  \, - \, \dfrac{|\alpha| \, \beta^2}{4} \, (D_{2,+} u_{0,0}^n)^2 
\, - \, (|\alpha|+|\beta|) \, \dfrac{(\alpha^2+\beta^2)}{8} \, (D_{1,+} D_{2,+} u_{0,0}^n)^2 \notag \\
& + \dfrac{(\alpha^2+\beta^2)}{2} \, \Big( (D_{1,+} u_{0,0}^n)^2 \, + \, (D_{2,+} u_{0,0}^n)^2 \Big) \, ,\label{estim-bilan-corner}
\end{align}
where the first line on the right-hand side corresponds to a dissipative contribution and the only two terms with a bad sign 
are collected in the second line. We shall see in the concluding argument how to absorb these terms.
\bigskip

\underline{Step 3 (the boundary contribution).} We consider the quantity $\mathcal{B}_1$ defined in \eqref{defbilanbord} and absorb 
again some cross terms. The first cross term in the second line on the right-hand side of \eqref{defbilanbord} is estimated by Young's 
inequality:
$$
\beta^2 \, \left| \sum_{j \ge 1} u_{j,0}^n \, D_{2,+} u_{j,0}^n \right| \, \le \, 
\dfrac{3 \, |\beta|}{4} \, \sum_{j \ge 1} (u_{j,0}^n)^2 \, + \, \dfrac{|\beta|^3}{3} \, \sum_{j \ge 1} (D_{2,+} u_{j,0}^n)^2 \, .
$$
Moreover, in the third line on the right-hand side of \eqref{defbilanbord}, we use the inequality $|\beta| \ge \beta^2$ to simplify the 
term that involves the sum of the $(D_{1,+} D_{2,+} u_{j,0}^n)^2$. At last, in the first line on the right-hand side of \eqref{defbilanbord}, 
we use the inequality $(1-|\beta|)^2 \ge 1/2$ that holds as long as the parameter $\varepsilon$ in \eqref{restrictionCFL} is chosen small 
enough. We thus get our first estimate:
\begin{align}
\mathcal{B}_1 \, \le& \, - \, \dfrac{|\beta|}{4} \, \sum_{j \ge 1} \, (u_{j,0}^n)^2 \, - \, \dfrac{|\beta|^3}{6} \, \sum_{j \ge 1} (D_{2,+} u_{j,0}^n)^2 
\, - \, \dfrac{\alpha^2 \, |\beta|}{2} \, \sum_{j \ge 1} (D_{1,+} u_{j,0}^n)^2 \, - \, \dfrac{\alpha^2}{16} \, \sum_{j \ge 1} (\Delta_1 u_{j,0}^n)^2 \notag \\
& \, - \, \dfrac{\alpha^2 \, \beta^2}{8} \, \sum_{j \ge 1} (D_{2,+} \Delta_1 u_{j,0}^n)^2 
\, - \, \dfrac{(\alpha^2+\beta^2)}{8} \, \sum_{j \ge 1} (D_{1,+} D_{2,+} u_{j,0}^n)^2
\, - \, |\alpha| \, \beta^2 \, \sum_{j \ge 1} D_{1,0} u_{j,0}^n \, D_{2,+} u_{j,0}^n \notag \\
& \, - \, \dfrac{(\alpha^2+\beta^2)}{4} \, \sum_{j \ge 1} D_{1,+} u_{j,0}^n \, D_{1,+} D_{2,+} u_{j,0}^n 
\, + \, |\alpha| \, \dfrac{(\alpha^2+\beta^2)}{4} \, \sum_{j \ge 1} D_{1,0} u_{j,0}^n \, D_{2,+} \Delta_1 u_{j,0}^n \label{bilanbord1} \\
& \, - \, \dfrac{\alpha^2 \, (\alpha^2 + \beta^2)}{8} \, \sum_{j \ge 1} (\Delta_1 u_{j,0}^n) \, D_{2,+} \Delta_1 u_{j,0}^n 
\, + \, |\alpha|^3 \, |\beta| \, \sum_{j \ge 1} (\Delta_1 u_{j,0}^n) \, D_{2,+} D_{1,0} u_{j,0}^n \, .\notag
\end{align}

We use again Young's inequality for the first cross term in the third line on the right-hand side of \eqref{bilanbord1}:
$$
\dfrac{(\alpha^2+\beta^2)}{4} \, \left| \sum_{j \ge 1} D_{1,+} u_{j,0}^n \, D_{1,+} D_{2,+} u_{j,0}^n \right| \, \le \, 
\dfrac{(\alpha^2+\beta^2)}{4} \, \sum_{j \ge 1} (D_{1,+} u_{j,0}^n)^2 \, + \, 
\dfrac{(\alpha^2+\beta^2)}{16} \, \sum_{j \ge 1} (D_{1,+} D_{2,+} u_{j,0}^n)^2 \, ,
$$
and we now recall that the CFL parameters $\lambda,\mu$ are subject to the conditions \eqref{majorationCFL} where $M>0$ is a fixed 
constant. This means that the negative parameters $\alpha$ and $\beta$ satisfy $|\alpha| \le M \, |\beta|$ and $|\beta| \le M \, |\alpha|$ 
so we have:
\begin{multline*}
\dfrac{(\alpha^2+\beta^2)}{4} \, \left| \sum_{j \ge 1} D_{1,+} u_{j,0}^n \, D_{1,+} D_{2,+} u_{j,0}^n \right| \\
\le \, \dfrac{(1+M^2) \, |\beta|}{4} \, |\beta| \, \sum_{j \ge 1} (D_{1,+} u_{j,0}^n)^2 \, + \, 
\dfrac{(\alpha^2+\beta^2)}{16} \, \sum_{j \ge 1} (D_{1,+} D_{2,+} u_{j,0}^n)^2 \, .
\end{multline*}
Since we have $D_{1,+} u_{j,0}^n=u_{j+1,0}^n-u_{j,0}^n$, we get the (non-optimal !) estimate:
$$
\sum_{j \ge 1} (D_{1,+} u_{j,0}^n)^2 \, \le \, 4 \, \sum_{j \ge 1} (u_{j,0}^n)^2 \, ,
$$
and we thus see that choosing $\varepsilon$ in \eqref{restrictionCFL} possibly smaller (but the choice now depends on $M$ !), we get:
$$
\dfrac{(\alpha^2+\beta^2)}{4} \, \left| \sum_{j \ge 1} D_{1,+} u_{j,0}^n \, D_{1,+} D_{2,+} u_{j,0}^n \right| \, \le \, 
\dfrac{|\beta|}{12} \, \sum_{j \ge 1} (u_{j,0}^n)^2 \, + \, \dfrac{(\alpha^2+\beta^2)}{16} \, \sum_{j \ge 1} (D_{1,+} D_{2,+} u_{j,0}^n)^2 \, .
$$
Using this estimate in \eqref{bilanbord1}, we obtain:
\begin{align}
\mathcal{B}_1 \, \le& \, - \, \dfrac{|\beta|}{6} \, \sum_{j \ge 1} \, (u_{j,0}^n)^2 \, - \, \dfrac{|\beta|^3}{6} \, \sum_{j \ge 1} (D_{2,+} u_{j,0}^n)^2 
\, - \, \dfrac{\alpha^2 \, |\beta|}{2} \, \sum_{j \ge 1} (D_{1,+} u_{j,0}^n)^2 \, - \, \dfrac{\alpha^2}{16} \, \sum_{j \ge 1} (\Delta_1 u_{j,0}^n)^2 \notag \\
& \, - \, \dfrac{\alpha^2 \, \beta^2}{8} \, \sum_{j \ge 1} (D_{2,+} \Delta_1 u_{j,0}^n)^2 
\, - \, \dfrac{(\alpha^2+\beta^2)}{16} \, \sum_{j \ge 1} (D_{1,+} D_{2,+} u_{j,0}^n)^2 \notag \\
& \, - \, |\alpha| \, \beta^2 \, \sum_{j \ge 1} D_{1,0} u_{j,0}^n \, D_{2,+} u_{j,0}^n 
\, + \, |\alpha| \, \dfrac{(\alpha^2+\beta^2)}{4} \, \sum_{j \ge 1} D_{1,0} u_{j,0}^n \, D_{2,+} \Delta_1 u_{j,0}^n \label{bilanbord2} \\
& \, - \, \dfrac{\alpha^2 \, (\alpha^2 + \beta^2)}{8} \, \sum_{j \ge 1} (\Delta_1 u_{j,0}^n) \, D_{2,+} \Delta_1 u_{j,0}^n 
\, + \, |\alpha|^3 \, |\beta| \, \sum_{j \ge 1} (\Delta_1 u_{j,0}^n) \, D_{2,+} D_{1,0} u_{j,0}^n \, .\notag
\end{align}

We now deal with the first term in the third line on the right-hand side of \eqref{bilanbord2} and recall the relation $D_{1,0}=D_{1,+}
-\Delta_1/2$. We thus have:
$$
\, - \, |\alpha| \, \beta^2 \, \sum_{j \ge 1} D_{1,0} u_{j,0}^n \, D_{2,+} u_{j,0}^n \, = \, 
\, - \, |\alpha| \, \beta^2 \, \sum_{j \ge 1} D_{1,+} u_{j,0}^n \, D_{2,+} u_{j,0}^n \, + \, \dfrac{|\alpha| \, \beta^2}{2} \, 
\sum_{j \ge 1} \Delta_1 u_{j,0}^n \, D_{2,+} u_{j,0}^n \, .
$$
For the cross term with the product $D_{1,+} u_{j,0}^n \, D_{2,+} u_{j,0}^n$, we use again Young's inequality and \eqref{majorationCFL}:
\begin{align*}
|\alpha| \, \beta^2 \, \left| \sum_{j \ge 1} D_{1,+} u_{j,0}^n \, D_{2,+} u_{j,0}^n \right| \, &\le \, 
M \, |\beta|^3 \, \left| \sum_{j \ge 1} D_{1,+} u_{j,0}^n \, D_{2,+} u_{j,0}^n \right| \\
&\le \, \dfrac{|\beta|^3}{8} \, \sum_{j \ge 1} (D_{2,+} u_{j,0}^n)^2 \, + \, 2 \, M^2 \, |\beta|^3 \, \sum_{j \ge 1} (D_{1,+} u_{j,0}^n)^2 \\
&\le \, \dfrac{|\beta|^3}{8} \, \sum_{j \ge 1} (D_{2,+} u_{j,0}^n)^2 \, + \, 8 \, M^2 \, |\beta|^3 \, \sum_{j \ge 1} (u_{j,0}^n)^2 \, .
\end{align*}
We can therefore absorb this term on the right-hand side of \eqref{bilanbord2} by choosing $\varepsilon$ in \eqref{restrictionCFL} 
small enough. The argument for the cross term with the product $\Delta_1 u_{j,0}^n \, D_{2,+} u_{j,0}^n$ is similar and we leave it 
to the interested reader (it is important here to estimate $\beta$ in terms of $\alpha$, which is made possible by \eqref{majorationCFL}).

Using Young's inequality and possibly choosing $\varepsilon$ smaller in \eqref{restrictionCFL} (the choice still depends on $M$), 
we can also absorb the first cross term of the form $(\Delta_1 u_{j,0}^n) \, D_{2,+} \Delta_1 u_{j,0}^n$ in the fourth line of the right-hand 
side of \eqref{bilanbord2}. At this stage, by choosing $\varepsilon$ small enough and $(\alpha,\beta)$ that satisfy \eqref{restrictionCFL}, 
we can enforce the following estimate\footnote{The coefficients are not aimed to be optimal. Our main goal is to show here that all 
cross terms can be absorbed by choosing $\varepsilon$ small enough, but this requires to be able to estimate $\alpha$ by $\beta$ 
and $\beta$ by $\alpha$, as is made possible by \eqref{majorationCFL}.} for the boundary contribution $\mathcal{B}_1$:
\begin{align}
\mathcal{B}_1 \, \le& \, - \, \dfrac{|\beta|}{8} \, \sum_{j \ge 1} \, (u_{j,0}^n)^2 \, - \, \dfrac{|\beta|^3}{32} \, \sum_{j \ge 1} (D_{2,+} u_{j,0}^n)^2 
\, - \, \dfrac{\alpha^2 \, |\beta|}{2} \, \sum_{j \ge 1} (D_{1,+} u_{j,0}^n)^2 \, - \, \dfrac{\alpha^2}{32} \, \sum_{j \ge 1} (\Delta_1 u_{j,0}^n)^2 \notag \\
& \, - \, \dfrac{\alpha^2 \, \beta^2}{16} \, \sum_{j \ge 1} (D_{2,+} \Delta_1 u_{j,0}^n)^2 
\, - \, \dfrac{(\alpha^2+\beta^2)}{16} \, \sum_{j \ge 1} (D_{1,+} D_{2,+} u_{j,0}^n)^2 \label{bilanbord3} \\
& \, + \, |\alpha| \, \dfrac{(\alpha^2+\beta^2)}{4} \, \sum_{j \ge 1} D_{1,0} u_{j,0}^n \, D_{2,+} \Delta_1 u_{j,0}^n 
\, + \, |\alpha|^3 \, |\beta| \, \sum_{j \ge 1} (\Delta_1 u_{j,0}^n) \, D_{2,+} D_{1,0} u_{j,0}^n \, .\notag
\end{align}

For the very last two cross terms in the last line of the right-hand side of \eqref{bilanbord3}, we use the relation $D_{1,0}=D_{1,+}-\Delta_1/2$ 
in order to involve only the operators $D_{1,+}$, $D_{2,+}$ and $\Delta_1$. When we expand the second cross term, we have to deal with 
cross terms that involve the products:
$$
(\Delta_1 u_{j,0}^n) \, D_{1,+} D_{2,+} u_{j,0}^n \, \quad \quad (\Delta_1 u_{j,0}^n) \, D_{2,+} \Delta_1 u_{j,0}^n \, ,
$$
and for each of these two terms, we can apply the above argument that is based on Young's inequality (in order to absorb the ``worst'' 
square term) and choosing $\varepsilon$ sufficiently small (in order to absorb the remaining square term). The estimate \eqref{bilanbord3} 
thus yields, for instance, by choosing $\varepsilon$ small enough:
\begin{align}
\mathcal{B}_1 \, \le& \, - \, \dfrac{|\beta|}{8} \, \sum_{j \ge 1} \, (u_{j,0}^n)^2 \, - \, \dfrac{|\beta|^3}{32} \, \sum_{j \ge 1} (D_{2,+} u_{j,0}^n)^2 
\, - \, \dfrac{\alpha^2 \, |\beta|}{2} \, \sum_{j \ge 1} (D_{1,+} u_{j,0}^n)^2 \, - \, \dfrac{\alpha^2}{64} \, \sum_{j \ge 1} (\Delta_1 u_{j,0}^n)^2 \notag \\
& \, - \, \dfrac{\alpha^2 \, \beta^2}{32} \, \sum_{j \ge 1} (D_{2,+} \Delta_1 u_{j,0}^n)^2 
\, - \, \dfrac{(\alpha^2+\beta^2)}{32} \, \sum_{j \ge 1} (D_{1,+} D_{2,+} u_{j,0}^n)^2 \label{bilanbord4} \\
& \, + \, |\alpha| \, \dfrac{(\alpha^2+\beta^2)}{4} \, \sum_{j \ge 1} D_{1,+} u_{j,0}^n \, D_{2,+} \Delta_1 u_{j,0}^n 
\, - \, |\alpha| \, \dfrac{(\alpha^2+\beta^2)}{8} \, \sum_{j \ge 1} \Delta_1 u_{j,0}^n \, D_{2,+} \Delta_1 u_{j,0}^n \, .\notag
\end{align}

The analysis is almost complete. For the first cross term in the last line of the right-hand side of \eqref{bilanbord4}, we use the above argument 
to get:
$$
|\alpha| \, \dfrac{(\alpha^2+\beta^2)}{4} \, \left| \sum_{j \ge 1} D_{1,+} u_{j,0}^n \, D_{2,+} \Delta_1 u_{j,0}^n \right| \, \le \, 
\dfrac{\alpha^2 \, \beta^2}{64} \, \sum_{j \ge 1} (D_{2,+} \Delta_1 u_{j,0}^n)^2 \, + \, C(M) \, \beta^2 \, \sum_{j \ge 1} \, (u_{j,0}^n)^2 \, ,
$$
where the constant $C(M)$ only depends on $M$ (it is actually a polynomial quantity in $M$). The trouble comes with the very last term since 
the operator $\Delta_1$, unlike $D_{1,+}$ involves a symmetric stencil, and the control we have on $\Delta_1 u_{j,0}^n$ does not seem good 
enough at this stage. We shall therefore use the (non optimal) bound $\sum_{j\geq 1}(\Delta_1 \,u_{j,0}^n)^2\,\leq \,16\,\sum_{j\geq 0}(u^n_{j,0})^2$ 
to take advantage of the good control on $u_{j,0}^n$. Reproducing the above argument, we get the inequality:
$$
|\alpha| \, \dfrac{(\alpha^2+\beta^2)}{4} \, \left| \sum_{j \ge 1} \Delta_1 u_{j,0}^n \, D_{2,+} \Delta_1 u_{j,0}^n \right| \, \le \, 
\dfrac{\alpha^2 \, \beta^2}{64} \, \sum_{j \ge 1} (D_{2,+} \Delta_1 u_{j,0}^n)^2 \, + \, C(M) \, \beta^2 \, \sum_{j \ge 0} \, (u_{j,0}^n)^2 \, ,
$$
with a possibly larger constant $C(M)$ but the important thing is that the very last sum bears on the indices $\{ j \ge 0 \}$ and not only on 
$\{ j \ge 1 \}$.

As a conclusion for this third step, we have seen that we can choose $\varepsilon$ small enough (and the choice depends on $M$) such that, 
for CFL parameters that satisfy \eqref{restrictionCFL}, we get:
\begin{equation}
\label{estim-bilan-bord-1}
\mathcal{B}_1 \, \le \, C(M) \, \beta^2 \, (u_{0,0}^n)^2 \, - \, \dfrac{|\beta|}{9} \, \sum_{j \ge 1} \, (u_{j,0}^n)^2 \, .
\end{equation}
Here we forget about many of the nonpositive contributions in the right-hand side of \eqref{bilanbord4} in order to simplify the final estimate 
that is stated in Theorem \ref{thm1}. We recall that there is an analogous term $\mathcal{B}_2$ on the boundary $\{ j=0, k \ge 1 \}$ and this 
term satisfies the analogous estimate:
\begin{equation}
\label{estim-bilan-bord-2}
\mathcal{B}_2 \, \le \, C(M) \, \alpha^2 \, (u_{0,0}^n)^2 \, - \, \dfrac{|\alpha|}{9} \, \sum_{k \ge 1} \, (u_{0,k}^n)^2 \, .
\end{equation}
\bigskip

\underline{Conclusion.} From the estimate \eqref{decomp-final} and the three estimates \eqref{estim-bilan-int}, \eqref{estim-bilan-corner} 
(where we only keep the first dissipation term), \eqref{estim-bilan-bord-1} and \eqref{estim-bilan-bord-2}, we have:
\begin{multline*}
\| \, u^{n+1} \, \|^2 \, - \, \| \, u^n \, \|^2 \, + \,  \dfrac{\alpha^2}{8} \, \| \Delta_1 u^n \|_{\ell^2(\mathring{\mathbb{I}})}^2 
\, + \, \dfrac{\beta^2}{8} \, \| \Delta_2 u^n \|_{\ell^2(\mathring{\mathbb{I}})}^2 \\
\le \, C(M) \, (\alpha^2+\beta^2) \, (u_{0,0}^n)^2 \, + \, \dfrac{(\alpha^2+\beta^2)}{2} \, \Big( (D_{1,+} u_{0,0}^n)^2 \, + \, (D_{2,+} u_{0,0}^n)^2 \Big) 
\, - \, \dfrac{|\alpha|+|\beta|}{4} \, (u_{0,0}^n)^2 \\
\, - \, \dfrac{|\alpha|}{9} \, \sum_{k \ge 1} \, (u_{0,k}^n)^2 \, - \, \dfrac{|\beta|}{9} \, \sum_{j \ge 1} \, (u_{j,0}^n)^2 \, .
\end{multline*}
There are only three terms on the right-hand side with the ``wrong'' sign, but each of them can be absorbed by either one of the three terms 
with a negative sign since we have:
$$
 (D_{1,+} u_{0,0}^n)^2 \, + \, (D_{2,+} u_{0,0}^n)^2 \, \le \, 4 \, (u_{0,0}^n)^2 \, + \, 2 \, (u_{1,0}^n)^2 \, + \, 2 \, (u_{0,1}^n)^2 \, .
$$
Up to restricting again the parameter $\varepsilon$, we end up with the estimate:
$$
\| \, u^{n+1} \, \|^2 \, - \, \| \, u^n \, \|^2 \, + \, \dfrac{\alpha^2}{8} \, \| \Delta_1 u^n \|_{\ell^2(\mathring{\mathbb{I}})}^2 
\, + \, \dfrac{\beta^2}{8} \, \| \Delta_2 u^n \|_{\ell^2(\mathring{\mathbb{I}})}^2 
\, + \, \dfrac{|\alpha|}{10} \, \sum_{k \ge 0} \, (u_{0,k}^n)^2 \, + \, \dfrac{|\beta|}{10} \, \sum_{j \ge 0} \, (u_{j,0}^n)^2 \, \le \, 0 \, ,
$$
where we have incorporated the control of the corner value $u_{0,0}^n$ in the boundary sums for simplicity. The proof of Theorem \ref{thm1} 
is now complete.

\section{Discussion and numerical illustrations}
\label{part-num}

In this Section, we briefly discuss the result of Theorem \ref{thm1} and our strategy for its proof. If we compare with the analogous result 
in our previous article, that is \cite[Theorem 4.1]{BC1}, the main difference is that our stability estimate for the second order boundary and 
corner extrapolation does not cover the whole set of parameters $(\lambda \, a,\mu \, b)$ that satisfy the stability requirement for the Cauchy 
problem, namely the CFL condition:
\begin{equation}
\label{CFLcauchy}
(\lambda \, a)^2 \, + \, (\mu \, b)^2 \, \le \, \dfrac{1}{2} \, .
\end{equation}

In our final arguments, we have been rather crude in estimating the corner and boundary contributions in the decomposition \eqref{decomp-final}. 
Let us see whether there is a potential room for improvement. We first look at the corner contribution \eqref{defbilancoin}. The corner contribution 
$\mathcal{C}$ defined in \eqref{defbilancoin} is a quadratic form with respect to the vector $(u_{0,0}^n, D_{1,+} u_{0,0}^n, D_{2,+} u_{0,0}^n, 
D_{1,+} D_{2,+} u_{0,0}^n) \in \R^4$. Figure \ref{fig:corner-1} illustrates the set of parameters $(\lambda \, |a|,\mu \, |b|) \in [0,1]^2$ for which 
this quadratic form is negative definite (which is what we were aiming at in the proof of Theorem \ref{thm1}). The set of \emph{good} parameters, 
that is the parameters for which the quadratic form is definite negative and \eqref{CFLcauchy} holds, is depicted in yellow. The exterior of the 
ball \eqref{CFLcauchy} is depicted in dark blue. In the light blue region, \eqref{CFLcauchy} holds but the quadratic form is not negative definite. 
The latter region contains all small values of $\lambda \, |a|,\mu \, |b|$, which is reminiscent of our final estimate \eqref{estim-bilan-corner} and 
can be deduced by merely looking at the expression of $\mathcal{C}$ with either $a$ or $b$ equal to zero.

\begin{figure}[ht!]
  \centering
 \includegraphics[width=.4\textwidth]{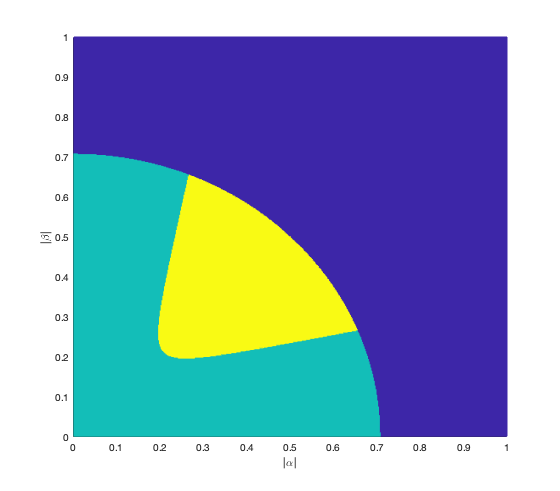}
  \caption{Negativity of the quadratic form associated with the corner contribution $\mathcal{C}$. In dark blue: the exterior of the ball. In yellow: 
  the parameters $(\lambda \, |a|,\mu \, |b|)$ for which \eqref{CFLcauchy} holds and the quadratic form is negative definite. In light blue: the 
  parameters $(\lambda \, |a|,\mu \, |b|)$ for which \eqref{CFLcauchy} holds and the quadratic form is not negative definite.}
  \label{fig:corner-1}
\end{figure}

The main problem in the analysis of the corner contribution $\mathcal{C}$ comes from the cross terms:
$$
D_{1,+} u_{0,0}^n \, D_{1,+} D_{2,+} u_{0,0}^n \, ,\quad \text{\rm and } \quad D_{2,+} u_{0,0}^n \, D_{1,+} D_{2,+} u_{0,0}^n \, .
$$
In our analysis, these terms are partly absorbed thanks to the boundary dissipation. If one now looks at the reduced corner contribution 
(where we omit the cross terms from which some difficulties arise, compare with \eqref{defbilancoin}):
\begin{align}
\widetilde{\mathcal{C}} \, :=& \, \left( |\alpha| \, |\beta| \, - \, \dfrac{|\alpha|+|\beta|}{2} \right) \, (u_{0,0}^n)^2 \notag \\
& - \, \left( \dfrac{|\alpha|^3}{4} \, + \, \dfrac{\alpha^2 \, |\beta|}{2} \right) \, (D_{1,+} u_{0,0}^n)^2 
\, - \, \left( \dfrac{|\beta|^3}{4} \, + \, \dfrac{|\alpha| \, \beta^2}{2} \right) \, (D_{2,+} u_{0,0}^n)^2 \notag \\
& \, - \, \dfrac{\alpha^2}{2} \, u_{0,0}^n \, D_{1,+} u_{0,0}^n \, - \, \dfrac{\beta^2}{2} \, u_{0,0}^n \, D_{2,+} u_{0,0}^n 
\, - \, \dfrac{|\alpha \, \beta|}{2} \, (|\alpha|+|\beta|) \, D_{1,+} u_{0,0}^n \, D_{2,+} u_{0,0}^n \label{defbilancoin'} \\
& \, - \, \dfrac{(\alpha^2+\beta^2)}{4} \, u_{0,0}^n \, D_{1,+} D_{2,+} u_{0,0}^n 
\, - \, \dfrac{3 \, (\alpha^2+\beta^2)}{16} \, (D_{1,+} D_{2,+} u_{0,0}^n)^2 \, ,\notag \\
&\, - \, (|\alpha|+|\beta|) \, \dfrac{(\alpha^2+\beta^2)}{8} \, (D_{1,+} D_{2,+} u_{0,0}^n)^2 
\, - \, \dfrac{(\alpha^2 + \beta^2)^2}{16} \, (D_{1,+} D_{2,+} u_{0,0}^n)^2 \, ,\notag
\end{align}
we can still try to identify numerically the region of parameters for which the associated quadratic form is negative definite. The result is 
shown in Figure \ref{fig:corner-2} with the same color scale as in Figure \ref{fig:corner-1}. The yellow region is far larger, which confirms 
that the above two cross terms are the core of the problem. Nevertheless, there remain very tiny portions near the axes, that is when 
either $\lambda \, |a|$ dominates $\mu \, |b|$ or the opposite, where the reduced quadratic form $\widetilde{\mathcal{C}}$ is not negative 
definite (this can also be seen by setting either $a$ or $b$ equal to zero). In full generality, it thus seems necessary to absorb part of the 
``bad'' terms in the corner contribution by part of the good terms in the boundary contributions.

\begin{figure}[ht!]
  \centering
 \includegraphics[width=.4\textwidth]{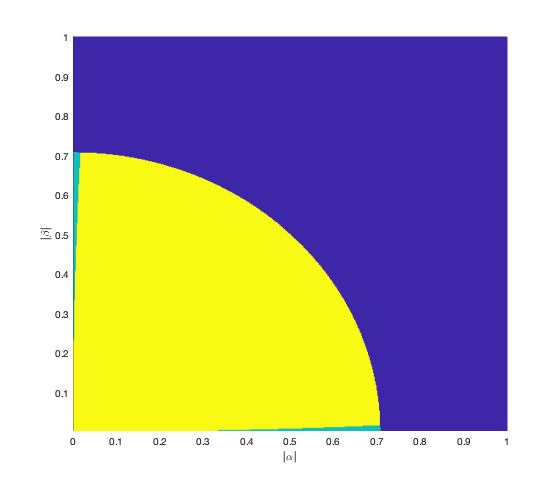}
  \caption{Negativity of the quadratic form associated with the reduced corner contribution $\widetilde{\mathcal{C}}$. In dark blue: the exterior 
  of the ball. In yellow: the parameters $(\lambda \, |a|,\mu \, |b|)$ for which \eqref{CFLcauchy} holds and the quadratic form is negative definite. 
  In light blue: the parameters $(\lambda \, |a|,\mu \, |b|)$ for which \eqref{CFLcauchy} holds and the quadratic form is not negative definite.}
  \label{fig:corner-2}
\end{figure}

We now examine the boundary contribution $\mathcal{B}_1$ in \eqref{defbilanbord}. Actually, we are going to simplify a little bit and consider 
the analogue of this term when extended to the whole set of integers $\Z$. In other words, we consider two sequences $u \in \ell^2(\Z;\R)$ 
($u$ being a placeholder for $u_{\cdot,0}^n$) and $v \in \ell^2(\Z;\R)$ ($v$ being a placeholder for the normal derivative $D_{2,+} u_{\cdot,0}^n$), 
and we consider the quantity:
\begin{align}
\widetilde{\mathcal{B}} \, :=& \, - \, |\beta| \, \sum_{j \in \Z} \, u_j^2 \, - \, \dfrac{|\beta|^3}{2} \, \sum_{j \in \Z} v_j^2 
\, - \, \dfrac{\alpha^2 \, |\beta|}{2} \, \sum_{j \in \Z} (D_{1,+} u_j)^2 \, - \, \dfrac{\alpha^2 \, (1-|\beta|)^2}{8} \, \sum_{j \in \Z} (\Delta_1 u_j)^2 \notag \\
& \, - \, \beta^2 \, \sum_{j \in \Z} u_j \, v_j \, - \, |\alpha| \, \beta^2 \, \sum_{j \in \Z} D_{1,0} u_j \, v_j 
\, - \, \dfrac{\alpha^2 \, \beta^2}{8} \, \sum_{j \in \Z} (\Delta_1 v_j)^2 \notag \\
& \, + \, |\alpha| \, \dfrac{(\alpha^2+\beta^2)}{4} \, \sum_{j \in \Z} D_{1,0} u_j \, \Delta_1 v_j 
\, - \, \dfrac{(1+|\beta|-\beta^2) \, (\alpha^2+\beta^2)}{8} \, \sum_{j \in \Z} (D_{1,+} v_j)^2 \label{defbilanbord'} \\
& \, - \, \dfrac{(\alpha^2+\beta^2)}{4} \, \sum_{j \in \Z} D_{1,+} u_j \, D_{1,+} v_j 
\, - \, \dfrac{\alpha^2 \, (\alpha^2 + \beta^2)}{8} \, \sum_{j \in \Z} \Delta_1 u_j \, \Delta_1 v_j 
\, + \, |\alpha|^3 \, |\beta| \, \sum_{j \in \Z} \Delta_1 u_j \, D_{1,0} v_j \, .\notag
\end{align}
After identifying the sequences $u$ and $v$ with square integrable piecewise constant functions on $\R$, we can apply Plancherel Theorem 
and obtain:
$$
\widetilde{\mathcal{B}} \, = \, \dfrac{1}{2 \, \pi} \, \int_\R \begin{pmatrix}
\hat{u}(\xi)\\
\hat{v}(\xi) \end{pmatrix}^* \, H(\xi) \, \begin{pmatrix}
\hat{u}(\xi)\\
\hat{v}(\xi) \end{pmatrix} \, {\rm d}\xi \, ,
$$
where the $2 \times 2$ Hermitian matrix $H(\xi)$ is defined by:
$$
H(\xi) \, := \, \begin{pmatrix}
-|\beta| (1+2\, \alpha^2 \, x) -2 \, \alpha^2 \, (1-|\beta|)^2 \, x^2 & c+i \, d \\
c-i \, d & -\dfrac{|\beta|^3}{2} -\dfrac{(1+|\beta|-\beta^2) \, (\alpha^2+\beta^2)}{2} \, x -2 \, \alpha^2 \, \beta^2 \, x^2 
\end{pmatrix} \, ,
$$
where we use the short notation $x:=\sin^2 (\xi/2)$, and the quantities $c$ and $d$ are defined by:
\begin{align*}
c \, := & \, - \, \dfrac{\beta^2}{2} \, - \, (\alpha^2+\beta^2) \, \dfrac{x}{2} \, - \, \alpha^2 \, (\alpha^2+\beta^2) \, x^2 \, ,\\
d \, := & \, \sin \xi \, \left( \dfrac{|\alpha| \, \beta^2}{2} \, + \, \dfrac{|\alpha| \, (\alpha^2+\beta^2)}{2} \, x -2 \, |\alpha|^3 \, |\beta| \, x \right) \, .
\end{align*}

If $H(\xi)$ is negative definite for any $\xi \in \R$, then we can get an upper bound for the boundary contribution $\widetilde{\mathcal{B}}$ 
with the ``good'' sign. The trace of $H(\xi)$ is easily shown to be negative for parameters $(\alpha,\beta)=(\lambda \, a,\mu \, b)$ that satisfy 
\eqref{CFLcauchy} and $\beta<0$. Thus $H(\xi)$ being negative definite is equivalent to showing that the determinant of $H(\xi)$ is positive 
for any $\xi \in \R$ (this determinant is shown to depend only on $x$ so there only remains a free parameter $x$ in the interval $[0,1]$). 
Figure \ref{fig:bord} shows the set of parameters $(\lambda \, |a|,\mu \, |b|)$ for which $H(\xi)$ is negative definite for any $\xi \in \R$ with 
the same color scale as in Figures \ref{fig:corner-1} and \ref{fig:corner-2}. The important point here is that for any parameters that satisfy 
the CFL condition \eqref{CFLcauchy}, the simplified boundary contribution $\widetilde{\mathcal{B}}$ does seem to provide some dissipation. 
However, we have not been able to derive the optimal scaling of this dissipation in terms of $\alpha$ and $\beta$ and it is therefore not clear 
whether, in a quarter-plane, the boundary terms may compensate for all the contributions at the corner that do not have the correct sign. 
This is left open for further studies.

\begin{figure}[ht!]
  \centering
 \includegraphics[width=.4\textwidth]{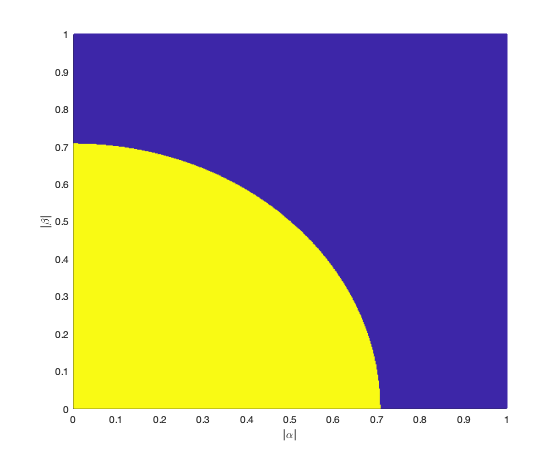}
  \caption{Negativity of the simplified boundary contribution $\widetilde{\mathcal{B}}$. In dark blue: the exterior 
  of the ball. In yellow: the parameters $(\lambda \, |a|,\mu \, |b|)$ for which \eqref{CFLcauchy} holds and the quadratic form is negative definite. 
  In light blue: the parameters $(\lambda \, |a|,\mu \, |b|)$ for which \eqref{CFLcauchy} holds and the quadratic form is not negative definite.}
  \label{fig:bord}
\end{figure}

\bibliographystyle{alpha}
\bibliography{BC}
\end{document}